\documentclass[12pt,american,english]{article}
\usepackage[T1]{fontenc}
\usepackage[latin9]{inputenc}
\usepackage{geometry}
\usepackage{babel}
\usepackage{amsmath}
\usepackage{amsthm}
\usepackage{amssymb}
\usepackage{stmaryrd}
\usepackage{graphicx}
\usepackage[all]{xy}
\usepackage[unicode=true,pdfusetitle,
 bookmarks=true,bookmarksnumbered=false,bookmarksopen=false,
 breaklinks=false,pdfborder={0 0 1},backref=false,colorlinks=false]
 {hyperref}

\makeatletter
\numberwithin{equation}{section}
\theoremstyle{plain}
\newtheorem{thm}{\protect\theoremname}[section]
\theoremstyle{remark}
\newtheorem*{acknowledgement*}{\protect\acknowledgementname}
\theoremstyle{definition}
\newtheorem{defn}[thm]{\protect\definitionname}
\theoremstyle{remark}
\newtheorem{notation}[thm]{\protect\notationname}
\theoremstyle{remark}
\newtheorem{rem}[thm]{\protect\remarkname}
\theoremstyle{plain}
\newtheorem{lem}[thm]{\protect\lemmaname}
\theoremstyle{plain}
\newtheorem{prop}[thm]{\protect\propositionname}
\theoremstyle{plain}
\newtheorem{cor}[thm]{\protect\corollaryname}
\theoremstyle{plain}
\newtheorem{assumption}[thm]{\protect\assumptionname}
\theoremstyle{definition}
\newtheorem{example}[thm]{\protect\examplename}
\theoremstyle{remark}
\newtheorem{claim}[thm]{\protect\claimname}

\theoremstyle{plain}

\usepackage{amssymb}
\usepackage{lmodern}
\date{}

\makeatother

\addto\captionsamerican{\renewcommand{\acknowledgementname}{Acknowledgement}}
\addto\captionsamerican{\renewcommand{\assumptionname}{Assumption}}
\addto\captionsamerican{\renewcommand{\claimname}{Claim}}
\addto\captionsamerican{\renewcommand{\corollaryname}{Corollary}}
\addto\captionsamerican{\renewcommand{\definitionname}{Definition}}
\addto\captionsamerican{\renewcommand{\examplename}{Example}}
\addto\captionsamerican{\renewcommand{\lemmaname}{Lemma}}
\addto\captionsamerican{\renewcommand{\notationname}{Notation}}
\addto\captionsamerican{\renewcommand{\propositionname}{Proposition}}
\addto\captionsamerican{\renewcommand{\remarkname}{Remark}}
\addto\captionsamerican{\renewcommand{\theoremname}{Theorem}}
\addto\captionsenglish{\renewcommand{\acknowledgementname}{Acknowledgement}}
\addto\captionsenglish{\renewcommand{\assumptionname}{Assumption}}
\addto\captionsenglish{\renewcommand{\claimname}{Claim}}
\addto\captionsenglish{\renewcommand{\corollaryname}{Corollary}}
\addto\captionsenglish{\renewcommand{\definitionname}{Definition}}
\addto\captionsenglish{\renewcommand{\examplename}{Example}}
\addto\captionsenglish{\renewcommand{\lemmaname}{Lemma}}
\addto\captionsenglish{\renewcommand{\notationname}{Notation}}
\addto\captionsenglish{\renewcommand{\propositionname}{Proposition}}
\addto\captionsenglish{\renewcommand{\remarkname}{Remark}}
\addto\captionsenglish{\renewcommand{\theoremname}{Theorem}}
\providecommand{\acknowledgementname}{Acknowledgement}
\providecommand{\assumptionname}{Assumption}
\providecommand{\claimname}{Claim}
\providecommand{\corollaryname}{Corollary}
\providecommand{\definitionname}{Definition}
\providecommand{\examplename}{Example}
\providecommand{\lemmaname}{Lemma}
\providecommand{\notationname}{Notation}
\providecommand{\propositionname}{Proposition}
\providecommand{\remarkname}{Remark}
\providecommand{\theoremname}{Theorem}

\begin{document}
\selectlanguage{american}%
\global\long\def\RR{\mathbb{R}}%
\global\long\def\CC{\mathbb{C}}%
\global\long\def\HH{\mathbb{H}}%
\global\long\def\NN{\mathbb{N}}%
\global\long\def\ZZ{\mathbb{Z}}%
\global\long\def\QQ{\mathbb{Q}}%

\global\long\def\gam{\Gamma}%

\global\long\def\z{\zeta}%
\global\long\def\e{\epsilon}%
\global\long\def\a{\alpha}%
\global\long\def\b{\beta}%
\global\long\def\ga{\gamma}%
\global\long\def\ph{\varphi}%
\global\long\def\om{\omega}%
\global\long\def\lm{\lambda}%
\global\long\def\dl{\delta}%
\global\long\def\s{\sigma}%
\foreignlanguage{english}{}
\global\long\def\little{\varepsilon}%

\selectlanguage{english}%

\global\long\def\exd#1#2{\underset{#2}{\underbrace{#1}}}%
\global\long\def\exup#1#2{\overset{#2}{\overbrace{#1}}}%

\global\long\def\leb#1{\operatorname{Leb(#1)}}%

\global\long\def\haar{\operatorname{Haar}}%

\global\long\def\norm#1{\left\Vert #1\right\Vert }%

\global\long\def\Onevec{\underline{1}}%
\global\long\def\One{\mathbf{1}}%

\global\long\def\porsmall{\prec}%
\global\long\def\porbig{\succ}%

\global\long\def\diffeo{\simeq}%

\global\long\def\lra{\longrightarrow}%
\global\long\def\del{\partial}%
\global\long\def\dprime{\prime\prime}%

\global\long\def\manifold{\mathcal{M}}%
\global\long\def\base{\mathbf{B}}%

\global\long\def\sbgrp{H}%

\global\long\def\transpose{\mbox{t}}%

\global\long\def\norm#1{\left\Vert #1\right\Vert }%
\global\long\def\lilnorm#1{\Vert#1\Vert}%

\global\long\def\brac#1{(#1)}%
\global\long\def\vbrac#1{|#1|}%
\global\long\def\sbrac#1{[#1]}%
\global\long\def\dbrac#1{\langle#1\rangle}%
\global\long\def\cbrac#1{\{#1\}}%

\global\long\def\gl#1{\operatorname{GL}_{#1}}%

\global\long\def\sl#1{\operatorname{SL}_{#1}}%

\global\long\def\so#1{\operatorname{SO}_{#1}}%

\global\long\def\ort#1{\operatorname{O}_{#1}}%

\global\long\def\pgl#1{\operatorname{PGL}_{#1}}%

\global\long\def\po#1{\operatorname{PO}_{#1}}%

\global\long\def\Lat{\Gamma}%
\global\long\def\disgrp{\gam}%
\global\long\def\wc{Q}%

\global\long\def\prim{\operatorname{prim}}%

\global\long\def\diag#1{\operatorname{diag}\brac{#1}}%

\global\long\def\rank#1{\operatorname{rank}\left(#1\right)}%

\global\long\def\vol{\operatorname{vol}}%

\global\long\def\prob{\operatorname{prob}}%

\global\long\def\covol#1{\operatorname{covol}\brac{#1}}%

\global\long\def\sym{\mbox{Sym}}%

\global\long\def\sp#1#2{\mbox{span}_{#1}\brac{#2}}%

\global\long\def\id{\operatorname{id}}%

\global\long\def\idmat#1{\operatorname{I}_{#1}}%

\global\long\def\comp#1{\operatorname{#1}^{c}}%

\global\long\def\trunc#1#2{#1^{#2}}%

\global\long\def\topindex{q}%

\global\long\def\nbhd#1#2{\mathcal{O}_{#1}^{#2}}%

\global\long\def\symfund#1{F_{#1}}%
\global\long\def\groupfund#1{\widetilde{F_{#1}}}%

\global\long\def\symfundrec#1#2{F_{#1}^{\left(#2\right)}}%
\global\long\def\groupfundrec#1#2{\widetilde{F_{#1}}^{\left(#2\right)}}%

\global\long\def\Gset{\mathcal{B}}%
\global\long\def\latset{\Psi}%
\global\long\def\pairset{\varXi}%
\global\long\def\symset{\mathcal{E}}%
\global\long\def\factorset{\mathcal{F}}%
\global\long\def\groupfacset{\widetilde{\factorset}}%
\global\long\def\groupset{\widetilde{\symset}}%
\global\long\def\sphereset{\Phi}%
\global\long\def\domN{D}%

\global\long\def\funddom{\Omega}%

\global\long\def\parby#1#2{#1_{#2}}%

\global\long\def\ball#1{B_{#1}}%

\global\long\def\GIcomp{S}%
\global\long\def\cube{\square}%

\global\long\def\sphere#1{\mathbb{S}^{#1}}%

\global\long\def\gras#1{\operatorname{Gr}(#1)}%

\global\long\def\latspace#1{\mathcal{L}_{#1}}%

\global\long\def\unilatspace#1{\mathcal{L}_{#1}}%

\global\long\def\shapespace#1{\mathcal{X}_{#1}}%

\global\long\def\pairspace#1{\mathcal{P}_{#1}}%

\global\long\def\flagspace#1{\mathcal{F}_{#1}}%

\global\long\def\svec{\underline{s}}%
\global\long\def\Svec{\underline{S}}%

\global\long\def\sumS{\mathbf{S}}%
\global\long\def\sums{\mathbf{s}}%

\global\long\def\wvec{\underline{w}}%
\global\long\def\Wvec{\underline{W}}%

\global\long\def\sumW{\mathbf{W}}%
\global\long\def\sumw{\mathbf{w}}%

\global\long\def\roundo{r}%

\global\long\def\errexp{\tau}%

\global\long\def\lat{\Lambda}%
\global\long\def\qlat{L}%
\global\long\def\latfull{\Delta}%

\global\long\def\cov{X}%

\global\long\def\shape#1{\operatorname{shape}\left(#1\right)}%

\global\long\def\unilat#1{\left[#1\right]}%

\global\long\def\unisimlat#1{\left\llbracket #1\right\rrbracket }%

\global\long\def\perpen#1{#1^{\perp}}%
\global\long\def\perpeng#1#2{#1^{\perp,#2}}%

\global\long\def\factor#1{#1^{\pi}}%
 
\global\long\def\factorg#1#2{#1^{\pi,#2}}%

\global\long\def\dual#1{#1^{*}}%

\global\long\def\latlast#1#2{#1^{\underleftarrow{#2}}}%

\global\long\def\based{\mathbf{D}}%
\global\long\def\basec{\mathbf{C}}%
\global\long\def\halfK#1{\Upsilon(#1)}%

\global\long\def\pair#1{\operatorname{pair}\left(#1\right)}%

\title{Equidistribution of primitive lattices in $\RR^{n}$}
\author{Tal Horesh\thanks{IST Austria, \texttt{tal.horesh@ist.ac.at}.} \and
Yakov Karasik\thanks{Technion, Israel, \texttt{theyakov@gmail.com}.}}
\maketitle
\begin{abstract}
We count primitive lattices of rank $d$ inside $\ZZ^{n}$ as their
covolume tends to infinity, with respect to certain parameters of
such lattices. These parameters include, for example, the subsapce
that a lattice spans, namely its projection to the Grassmannian; its
homothety class; and its equivalence class modulo rescaling and rotation,
often referred to as a \emph{shape}. We add to a prior work of Schmidt
by allowing sets in the spaces of parameters that are general enough
to conclude joint equidistribution of these parameters. In addition
to the primitive $d$-lattices $\lat$ themselves, we also consider
their orthogonal complements in $\ZZ^{n}$, $\perpen{\lat}$, and
show that the equidistribution occurs jointly for $\lat$ and $\perpen{\lat}$.
Finally, our asymptotic formulas for the number of primitive lattices
include an explicit error term. 

\tableofcontents{}
\end{abstract}

\section{Introduction\label{sec: Introduction}}

The aim of this paper is to extend classical counting and equidistribution
results for \emph{primitive vectors} to their higher rank counterparts:
\emph{primitive lattices}. A primitive vector is an $n$-tuple of
integers $\brac{a_{1},\ldots,a_{n}}$ with $\gcd\brac{a_{1},\ldots,a_{n}}=1$,
and the the set of primitive vectors in $\RR^{n}$ is denoted by $\ZZ_{\prim}^{n}$.
We can associate to each vector $0\neq v\in\RR^{n}$ the discrete
subgroup that it spans, $\ZZ v$; following this logic, a rank $d$
($1\leq d\leq n$) analog for a vector is  a \emph{lattice of rank
$d$} in $\RR^{n}$, namely
\[
\lat=\ZZ v_{1}\oplus\cdots\oplus\ZZ v_{d}
\]
where $v_{1},\ldots,v_{d}\in\RR^{n}$ are linearly independent. We
will refer to it briefly as a \emph{$d$-lattice}. We say that a $d$-lattice
$\lat$ is \emph{integral} if $\lat\subset\ZZ^{n}$, and \emph{primitive}
if $\lat=V\cap\ZZ^{n}$, where $V$ is a $d$-dimensional rational
subspace of $\RR^{n}$. For example, a primitive $1$-lattice is simply
all the integral points on a rational line, or, equivalently, it is
$\ZZ v$ where $v$ is a primitive vector. 

Questions about counting  primitive vectors date back to the days
of Gauss and Dirichlet, e.g.\ with the \emph{Gauss Class Number problem}.
In the 20th century, questions about equidistribution of integral
vectors began to arise, with the principal example being \emph{Linnik-type
problems} \cite{Linnik_68,Erdos_Hall_99,Duke_03,Duke_07,ELMV_11,Benoist_Oh_12,EMV_13}.
These questions and others generalize naturally to  primitive lattices,
as we now turn to describe. 

\paragraph*{The Primitive Circle Problem.}

The well known Gauss Circle Problem concerns the asymptotic number
of integral vectors up to (euclidean) norm $X>0$. The analogous question
for primitive vectors, namely the asymptotic amount of primitive vectors
up to norm $X$, is often referred to as the \emph{primitive circle
problem} \cite{Nowak88,ZC99,Wu02}. In lattices, the role of a norm
is played by the \emph{covolume}: the covolume of $\lat$, denoted
$\covol{\lat}$, is the volume of a fundamental paralelopiped for
$\lat$ in the linear space 
\[
V_{\lat}:=\sp{\RR}{\lat}.
\]
Thus, the \textbf{\emph{primitive circle problem for lattices}} is
to estimate the asymptotics of
\begin{equation}
\#\{\text{primitive \ensuremath{d}-lattices in \ensuremath{\RR^{n}} of covolume up to \ensuremath{X}}\}\label{eq: Primitive Circle Problem}
\end{equation}
as $X\to\infty$. Note that for $1$-lattices, the notions of norm
and covolume coincide: $\covol{\ZZ v}=\norm v$, hence when $d=1$
the above recovers the ``original'' primitive circle problem. Schmidt
\cite{Schmidt_68} showed that the amount in (\ref{eq: Primitive Circle Problem})
equals 
\[
c_{d,n}X^{n}+O(X^{n-\max\{\frac{1}{d},\frac{1}{n-d}\}}),
\]
where 
\[
c_{d,n}=\frac{1}{n}{n \choose d}\cdot\frac{\prod_{i=n-d-1}^{n}\leb{\ball i}}{\prod_{j=1}^{d}\leb{\ball j}}\cdot\frac{\prod_{i=2}^{d}\zeta\left(i\right)}{\prod_{j=n-d+1}^{n}\zeta\left(j\right)},
\]
$\text{Leb}$ is the Lebesgue measure, and $\ball i$ is the unit
ball in $\RR^{i}$. We remark that the optimal exponent in the error
term of the circle problem (primitive or not) is  established only
in dimensions $n\geq4$. As far as the authors are aware, this ($d=1$,
$n\geq4$) is the only case where an optimal error exponent is known
for the lattices circle problem (primitive or not). 

\paragraph*{Linnik-type problems.}

This is a unifying name for questions on the distribution of the projections
of integral vectors to the unit sphere, i.e. of $v/\norm v$ when
$v\in\ZZ^{n}\text{ or \ensuremath{\ZZ_{\prim}^{n}}}$.  Viewing the
unit sphere as the space of oriented lines in $\RR^{n}$, the analogous
object when considering $d$-lattices would be the \emph{Grassmannian}
of oriented $d$-dimensional subspaces in $\RR^{n}$, denoted $\gras{d,n}$
(and defined explicitly in Section \ref{sec: Background on Lattices}).
Accordingly, we will view our lattices as carrying an orientation,
which simplifies our discussion on the technical level but has no
effect on the results. In particular, the two-to-one correspondence
between primitive vectors and primitive $1$-lattices (arising from
the fact that $v$ and $-v$ span the same lattice) becomes a one-to-one
correspondence between primitive vectors and oriented primitive $1$-lattices.
The \textbf{\emph{average Linnik problem for primitive lattices}}
is to study the distribution of the (oriented) spaces $V_{\lat}$
in $\gras{d,n}$ as $\covol{\lat}\leq X\to\infty$. Note that $V_{\lat}$
are exactly the rational subspaces in $\gras{d,n}$. 

\paragraph*{Shapes of orthogonal lattices.}

More recently, with the rise of dynamical approaches in number theory,
another type of equidistribution questions for primitive vectors arose.
To a primitive vector $v$ we associate the $(n-1)$-lattice $\perpen v\cap\ZZ^{n}$,
referred to as as the \emph{orthogonal lattice} of $v$, where $\perpen v$
is the orthogonal hyperplane to $v$. Several recent papers (by Marklof
\cite{Marklof_10}, Aka Einsiedler and Shapira \cite{AES_16A,AES_16B},
Einsiedler, Mozes, Shah and Shapira \cite{EMSS_16}, Einsiedler R�hr
and Wirth \cite{ERW17}) studied the equidistribution of \emph{shapes}
of the orthogonal lattices to primitive vectors as their norm tends
to infinity, where the shape of a lattice is its similarity class
modulo rotation and homothety. The space of shapes of $d$-lattices
is a double coset space of $\sl d(\RR)$, denoted $\shapespace d$
and defined explicitly in Section \ref{sec: Background on Lattices},
and the aforementioned papers show (among other things) that the shapes
of the orthogonal lattices to $v\in\ZZ_{\prim}^{n}$ equidistribute
in $\shapespace{n-1}$ as $\norm v\to\infty$ w.r.t.\ a uniform measure
arriving from the Haar measure on $\sl{n-1}(\RR)$. In the works of
Einsiedler et al.\, they in fact show that the shapes of $\perpen v\cap\ZZ^{n}$
equidistribute in $\shapespace{n-1}$ \emph{jointly} with the directions
$v/\norm v$ in $\sphere{n-1}$.

Just like for (primitive) vectors, orthogonal lattices can be defined
for (primitive) lattices as well: for a primitive $d$-lattice $\lat$,
we let:
\[
\perpen{\lat}:=\perpen{V_{\lat}}\cap\ZZ^{n},
\]
where $\perpen{V_{\lat}}$ is the orthogonal complement of $V_{\lat}$.
Note that $\perpen{\lat}$ is primitive by definition, and that it
has rank $n-d$. Also note that $\lat\mapsto\perpen{\lat}$ defines
a bijection between primitive lattices of ranks $d$ and $n-d$.
This bijetion extends to a bijection between oriented lattices, with
a natural choice of orientation on the orthogonal lattice (Definition
\ref{def: orientation in orthogonal}). 

One could then ask about the equidistribution of shapes of the orthogonal
lattices $\perpen{\lat}$ to primitive lattices $\lat$, where the
one dimensional case $\lat=\ZZ v$ recovers the question studied in
the aforementioned papers about the equidistribution of shapes of
$\perpen v\cap\ZZ^{n}$. 

Equidistribution of a sequence in a finite-volume space can be deduced
from counting in ``sufficiently general'' subsets of this space.
Indeed, we will count in subsets that have \emph{controlled boundary}
(Definition \ref{def: BCS}).

\begin{thm}[A consequence of Theorems \ref{thm: A counting} and \ref{thm: counting orthogonal lattices}]
\label{thm: Intro}Assume that $\sphereset\subseteq\gras{d,n}$ and
$\symset\times\factorset\subseteq\shapespace d\times\shapespace{n-d}$
are BCS's. Then, the number of primitive $d$-lattices $\lat$ with
$\covol{\lat}\leq X$, $V_{\lat}\in\sphereset$ and $\brac{\shape{\lat},\shape{\perpen{\lat}}}\in\symset\times\factorset$
is 
\[
\frac{\vol\brac{\symset}\vol\brac{\factorset}\vol\brac{\sphereset}}{n\cdot\prod_{i=2}^{n}\zeta\left(i\right)}\halfK d\halfK{n-d}\cdot X^{n}+O_{\e}\brac{X^{n-\frac{\errexp_{n}n}{\b}+\e}}
\]
for every $\e>0$, where $\errexp_{n}=\left\lceil \left(n-1\right)/2\right\rceil /4n^{2}$,
\[
\halfK d:=\frac{d!\prod_{i=2}^{d}\leb{B_{i}}}{\left[\so d\left(\RR\right):Z\left(\so d\left(\RR\right)\right)\right]}
\]
(in the denominator is the index of the center of $\so d\left(\RR\right)$,
which is either $1$ or $2$), and 
\[
\b=\begin{cases}
2\left(n-d-1\right)\left(n^{2}-1\right)+n^{2} & \text{when \ensuremath{\symset} is bounded}\\
2\left(d-1\right)\left(n^{2}-1\right)+n^{2} & \text{when \ensuremath{\factorset} is bounded}\\
2\left(\max\left\{ d,n-d\right\} -1\right)\left(n^{2}-1\right)+n^{2} & \text{when neither of \ensuremath{\symset,\factorset} is bounded}\\
1 & \text{when both \ensuremath{\symset,\factorset} are bounded}
\end{cases}.
\]
\end{thm}

In the above theorem, $\vol$ stands for a uniform Radon measure (that
is unique up to multiplication by a positive scalar and will be defined
in Section \ref{sec: Background on Lattices}), so it implies the
joint uniform distribution of the directions, shapes, and shapes of
the orthogonal lattices of primitive lattices as their covolume tends
to infinity. In particular, it is quite interesting to observe that
the shapes of $\lat$ and of $\perpen{\lat}$ are independent parameters,
meaning that there is no way to know the shape of $\perpen{\lat}$
given the shape of $\lat$, even though the latter lattice determines
the first. 

We remark that the equidistribution of shapes of primitive $d$-lattices
was shown in a non-quantitative form by Schmidt in \cite{Schmidt_98}.

\paragraph{Counting $d$-lattices: Beyond shapes.}

There exists a wide body of work on equidistribution problems in the
space 
\[
\latspace n:=\sl n(\RR)/\sl n(\ZZ),
\]
which is the space of unimodular (i.e.\ with covolume one and positive
orientation) full lattices in $\RR^{n}$, as well as in the space
$\shapespace n$ of their shapes. The restriction to covolume one
is due to the fact that the space $\latspace n$ (and therefore $\shapespace n$)
has finite volume, while the space of all lattices, $\gl n(\RR)/\gl n(\ZZ)$,
does not. Comparing the spaces $\latspace n$ and $\shapespace n$,
the space $\latspace n$ naturally contains ``more information''
than $\shapespace n$, which is obtained by modding $\latspace n$
by rotations. This brings up the question of whether one can define
a space of ``unimodular $d$-lattices in $\RR^{n}$'' so as to consider
the unimodular $d$-lattices in $\RR^{n}$ without modding out by
rotations. In Section \ref{sec: Background on Lattices} we introduce
two such spaces, which are homogeneous spaces of $\sl n(\RR)$, but
do not carry an $\sl n(\RR)$-invariant measure. However, they do
support a natural measure which is invariant under a direct product
of two subgroups of $\sl n(\RR)$. We will prove a stronger statement
than Theorem \ref{thm: Intro}, namely Theorem \ref{thm: A counting},
where we count primitive $d$-lattices according to their projections
to these more refined spaces, from which we conclude Theorem \ref{thm: Intro},
of counting primitive $d$-lattices according to their projections
to $\shapespace d$. 

\paragraph{Organization of the paper.}

The paper is organized as follows: In Section \ref{sec: Background on Lattices}
we define all the spaces and notions related to lattices that appear
throughout this work. In Section \ref{sec: Counting Lattices} we
state Theorem \ref{thm: A counting}, and explain how it implies Theorem
\ref{thm: Intro}; then the rest of the paper is devoted to proving
Theorem \ref{thm: A counting}. In Section \ref{sec: RI coordinates}
we define coordinates on $\sl n(\RR)$ that are a refinement of the
Iwasawa coordinates, and study their properties. In Section \ref{sec: Fundamental domains}
we identify subsets of $\sl n(\RR)$ that are sets of representatives
for all the quotient spaces that appear in Theorem \ref{thm: A counting}.
In Section \ref{sec: Integral matrices representing primitive vectors},
we associate to each primitive $d$-lattice a unique element in $\sl n(\ZZ)$,
transferring the problem of counting primitive $d$-lattices to the
problem of counting points of $\sl n(\ZZ)$ in a family of increasing
sets in $\sl n(\RR)$. These sets are not compact, even though the
number of lattice points in each of the sets is finite. In Section
\ref{sec: Counting lattices up the cusp}, we reduce to counting in
a family of compact subsets, obtaining a ``classical'' counting
lattice point problem in $\sl n(\RR)$. In Section \ref{sec: Almost  a proof}
we state the solution of this counting problem (Proposition \ref{thm: Counting with S(T) and W(T)}),
and prove Theorem \ref{thm: A counting} based on it. In Section \ref{sec: Technical Proof}
we prove Proposition \ref{thm: Counting with S(T) and W(T)} using
a method developed by Gorodnik and Nevo in \cite{GN1}, for counting
lattice points in semi-simple Lie groups. In the appendix, we expand
further on the notions introduced in Section \ref{sec: Background on Lattices},
as some of them are new, and others are not new but possibly unfamiliar
to the reader.
\begin{acknowledgement*}
This work was done when both authors were visiting IST Austria. The
first author was being supported by EPRSC grant EP/P026710/1, and
the second author had a great time there and is grateful for the hospitality.
The appendix to this paper is largely based on a mini course the first
author had given at IST in February 2020. 
\end{acknowledgement*}

\section{\label{sec: Background on Lattices}Spaces of lattices}

We begin by defining explicitly the spaces $\shapespace d$ and $\gras{d,n}$
appearing in Theorem \ref{thm: Intro}. An oriented subspace on $\RR^{n}$
is a subspace with a sign attached, and the Grassmannian $\gras{d,n}$
is the set of all $d$-dimensional oriented subspaces of $\RR^{n}$.
It can also be defined as the following quotients 
\[
\gras{d,n}=\so n\left(\RR\right)/\left\{ \left[\begin{smallmatrix}\so d\left(\RR\right) & 0_{d,n-d}\\
0_{n-d,d} & \so{n-d}\left(\RR\right)
\end{smallmatrix}\right]\right\} =\sl n\left(\RR\right)/\left[\begin{smallmatrix}\sl d\left(\RR\right) & \RR^{d,n-d}\\
0_{n-d,d} & \sl{n-d}\left(\RR\right)
\end{smallmatrix}\right],
\]
where $P_{d}$ is the group of upper triangular matrices of determinant
$1$ with positive diagonal entries. A coset of $\so n$ (or $\sl n$)
represents an oriented $d$-dimensional subspace $V$ if the first
$d$ columns of the matrices in this coset span $V$ with the corresponding
orientation. 

Recall that the shape of (an oriented) lattice is its equivalent class
modulo homothety and rotation. The space of shapes of (oriented) $d$-lattices
is 
\[
\shapespace d=\so d\left(\RR\right)\backslash\sl d\left(\RR\right)/\sl d\left(\ZZ\right)\diffeo P_{d}/\sl d\left(\ZZ\right).
\]
Both $\gras{d,n}$ and $\shapespace d$ are equipped with natural
measures that decent from the Haar measure on ambient group, and are
unique up to a positive scalar \textendash{} for $\gras{d,n}$ it
is the measure that is invariant under $\so n\left(\RR\right)$, and
for $\shapespace d$ it is the measure that is obtained from the Haar
measure on $\sl n(\RR)$. The specific normalizations we choose for
these measures are given in Section \ref{subsec: measures on spaces}. 

We now proceed to define a space of $d$-lattices in $\RR^{n}$ that
encodes both their shapes, and their directions (i.e. their projections
to the Grassmannian). 

\subsection{Spaces of homothety classes of $d$-lattices.}

A \emph{unimodular} lattice is an oriented lattice with positive orientation
and covolume one. Recall our notation for the space of rank $n$ unimodular
lattices 
\[
\latspace n=\sl n(\RR)/\sl n(\ZZ),
\]
where a matrix in $\sl n\left(\RR\right)$ lies in the coset that
represents a full unimodular lattice in $\RR^{n}$ if its columns
span this lattice. As the space of shapes $\shapespace n$ is obtained
from $\latspace n$ by modding by $\so n$, the space $\latspace n$
is more refined, containing not only the information about the shape
of a lattice, but also about its position in $\RR^{n}$. To define
the analogous space for $d$-lattices in $\RR^{n}$, notice that since
$\latspace n$ consists of a unique representative from any equivalence
class of $n$-lattices in $\RR^{n}$ modulo homothety, one can identify
$\latspace n$ with the space of such equivalence classes. The space
of homothety classes of oriented $d$-lattices inside $\RR^{n}$ is
\[
\latspace{d,n}:=\sl n\left(\RR\right)/\left(\left[\begin{smallmatrix}\sl d\left(\ZZ\right) & \RR^{d,n-d}\\
0_{n-d\times d} & \sl{n-d}\left(\RR\right)
\end{smallmatrix}\right]\times\left\{ \left[\begin{smallmatrix}\a^{\frac{1}{d}}\idmat d & 0_{d\times n-d}\\
0_{n-d\times d} & \a^{-\frac{1}{n-d}}\idmat{n-d}
\end{smallmatrix}\right]:\a>0\right\} \right),
\]
where a matrix in $\sl n\left(\RR\right)$ lies in the coset that
represents an equivalence class of a $d$-lattice in $\RR^{n}$ if
and only if its first $d$ columns span a positive scalar multiplication
of this lattice, with the corresponding orientation. The need to mod
out by the block-scalar group follows from the fact that the first
$d$ columns of a matrix in $\sl n$ span a lattice that  is hardly
ever unimodular, so one needs to mod out by the covolume (hence an
element in this quotient space is an equivalence class of $d$-lattices
up to homothety). However, modding out the block-scalar group comes
with a price, which is that the space $\latspace{d,n}$ does not carry
an $\sl n(\RR)$-invariant measure. To fix this flaw, we claim (Remark
\ref{rem: true quotients L_d,n and P_d,n}) that 
\[
\latspace{d,n}\diffeo\so n\left(\RR\right)\left[\begin{smallmatrix}P_{d}(\RR) & 0_{d\times n-d}\\
0_{n-d\times d} & \idmat{n-d}
\end{smallmatrix}\right]/\left[\begin{smallmatrix}\sl d\left(\ZZ\right) & 0_{d\times n-d}\\
0_{n-d\times d} & \so{n-d}\left(\RR\right)
\end{smallmatrix}\right],
\]
and observe that, while the quotiented manifold  is not a group,
it is diffeomorphic to the group $\so n\left(\RR\right)\times\left[\begin{smallmatrix}P_{d}(\RR) & 0_{d\times n-d}\\
0_{n-d\times d} & \idmat{n-d}
\end{smallmatrix}\right]$. Therefore, $\latspace{d,n}$ carries a Radon measure that is invariant
under this group, and unique up to a positive scalar. This measure
is the natural one in the sense that, as we shall see below, it projects
to the uniform measures on $\latspace d$ and $\shapespace d$. With
respect to it, $\latspace{d,n}$ has finite volume (which wouldn't
have been the case if we hadn't modded out by the covolume, just like
in the case of $\latspace n$). 

\subsection{Factor lattices and the space of pairs. }

Recall that a lattice $\lat$ is primitive if it is of the form $\lat=\ZZ^{n}\cap V_{\lat}$;
this is equivalent to the fact that any basis of $\lat$ can be completed
to a basis of $\ZZ^{n}$.  Theorem \ref{thm: Intro} consists of
a joint equidistribution result for primitive lattices $\lat$ and
their orthogonal complements $\perpen{\lat}=\ZZ^{n}\cap\perpen{V_{\lat}}$.
It is a consequence of the stronger Theorem \ref{thm: A counting}
below, in which $\perpen{\lat}$ is replaced by another $\left(n-d\right)$-lattice
in $\perpen{V_{\lat}}$:
\begin{defn}
The \emph{factor lattice} $\factor{\lat}$ of a primitive $d$-lattice
$\lat$ is the projection of $\ZZ^{n}$ to $\perpen{V_{\lat}}$. 
\end{defn}

The factor lattice $\factor{\lat}$ is isometric to the quotient $\ZZ^{n}/\lat$
(Proposition \ref{prop: factor is quotient}), so one should think
of $\factor{\lat}$ as a realization of $\ZZ^{n}/\lat$ inside $\RR^{n}$.
Notice that, like $\perpen{\lat}$, $\factor{\lat}$ is a full lattice
inside $\perpen{V_{\lat}}$, and in particular is of rank $n-d$.
The relation between $\perpen{\lat}$ and $\factor{\lat}$ is that
they are \emph{dual} to one another (for the definition of dual lattices,
see \cite[I.5]{Cassels71}, or Appendix \ref{sec: Dual Lattices}
in the present paper. For the duality of $\perpen{\lat}$ and $\factor{\lat}$,
see Claim \ref{prop: dual+factor =00003D orthogonal}). It holds that
\[
\covol{\factor{\lat}}=\covol{\lat}^{-1}=\covol{\perpen{\lat}}^{-1}
\]
(\cite{Schmidt_67,Schmidt_68}; see also Proposition \ref{prop: factor covolume}
and Corollary \ref{cor: orthogonal covolume} in the Appendix). Since
we view $d$-lattices as carrying an orientation, we need to define
an orientation on $\factor{\lat}$ and $\perpen{\lat}$, which is
done as follows. 
\begin{defn}
\label{def: orientation in orthogonal}Let $L$ be a full lattice
in $\perpen{V_{\lat}}$. A basis $\basec$ for $L$ is positively
oriented if $\det(\base|\basec)=1$ for a positively oriented basis
$\base$ of $\lat$.
\end{defn}

The last space we introduce is the space of pairs of oriented lattices
$\brac{\lat,L}$ such that (i) $\lat$ is a $d$-lattice, (ii) $L$
is a full lattice in $\perpen{V_{\lat}}$ (hence is of rank $n-d$
and its orientation is given in Definition \ref{def: orientation in orthogonal}),
and (iii) $\covol{\lat}\covol{\qlat}=1$. In fact, it is the space
of homothety classes of such pairs, where $\brac{\lat,L}$ is equivalent
to $\brac{\a^{\frac{1}{d}}\lat,\a^{-\frac{1}{n-d}}L}$ for every $\a>0$.
This space is given as the quotient 
\[
\text{\ensuremath{\pairspace{d,n}}}:=\sl n\left(\RR\right)/\left(\left[\begin{smallmatrix}\sl d\left(\ZZ\right) & \RR^{d,n-d}\\
0_{n-d\times d} & \sl{n-d}\left(\ZZ\right)
\end{smallmatrix}\right]\times\left\{ \left[\begin{smallmatrix}\a^{\frac{1}{d}}\idmat d & 0_{d\times n-d}\\
0_{n-d\times d} & \a^{-\frac{1}{n-d}}\idmat{n-d}
\end{smallmatrix}\right]:\a>0\right\} \right).
\]
Note the similarity to the definition of $\latspace{d,n}$ above;
here also, modding out the block-scalar group results in having to
$\sl n(\RR)$-invariant measure, whih can be fixed by observing (Remark
\ref{rem: true quotients L_d,n and P_d,n}) that 
\[
\pairspace{d,n}\diffeo\so n\left(\RR\right)\left[\begin{smallmatrix}P_{d}(\RR) & 0_{d\times n-d}\\
0_{n-d\times d} & P_{n-d}
\end{smallmatrix}\right]/\left[\begin{smallmatrix}\sl d\left(\ZZ\right) & 0_{d\times n-d}\\
0_{n-d\times d} & \sl{n-d}\left(\ZZ\right)
\end{smallmatrix}\right],
\]
and therefore $\pairspace{d,n}$ carries a Radon measure that is invariant
under the group $\so n\left(\RR\right)\times\left[\begin{smallmatrix}P_{d}(\RR) & 0_{d\times n-d}\\
0_{n-d\times d} & P_{n-d}
\end{smallmatrix}\right]$, and is unique up to a positive scalar. 

More details on dual lattices and factor lattices can be found in
the appendix.

\subsection{\label{subsec: measures on spaces}Relation between the spaces of
lattices and their measures}

While the spaces $\unilatspace d$, $\shapespace d$ and $\gras{d,n}$
are well known, the spaces $\latspace{d,n}$ and $\text{\ensuremath{\pairspace{d,n}}}$,
as far as we are aware, make their first appearance here. It therefore
seems appropriate to explain how $\latspace{d,n}$ and $\text{\ensuremath{\pairspace{d,n}}}$
add to the more ``familiar'' spaces $\unilatspace d$, $\shapespace d$
and $\gras{d,n}$. It is not hard to see that $\text{\ensuremath{\pairspace{d,n}}}$
projects canonically to $\latspace{d,n}$, which in turn projects
canonically to $\shapespace d$ and $\gras{d,n}$. In fact, one could
also map $\latspace{d,n}\to\unilatspace d$ and $\pairspace{d,n}\to\unilatspace{n-d}$,
but these projections are not cannonical, and depend on a choice of
coordinates on the $d$-dimensional (resp.\ $n-d$ dimensional) vector
spaces in $\RR^{n}$. Such a choice of coordinates will be introduced
in Section \ref{subsec: Introducing RI coordinates}, defining the
non-canonical identifications 
\begin{equation}
\begin{aligned}\latspace{d,n}\;\leftrightarrow\; & \;\unilatspace d\times\gras{d,n},\\
\pairspace{d,n}\;\leftrightarrow\; & \;\latspace{d,n}\times\unilatspace{n-d}.
\end{aligned}
\label{eq: identifications}
\end{equation}
The full picture is depicted in the following commutative diagram
of projections:

\begin{equation}
\xymatrix{\pairspace{d,n}\ar@{->>}[r]\ar@{->>}[d]\ar@{->>}@(dl,ul)[dd]\ar@{->>}[ddr]\ar@{->>}@(ur,ul)[rr] & \latspace{d,n}\ar@{->>}[r]\ar@{->>}[d]\ar@{->>}@(dr,ur)[dd] & \gras{d,n}\\
\unilatspace{n-d}\ar@{->>}[d] & \unilatspace d\ar@{->>}[d]\\
\shapespace{n-d} & \shapespace d
}
,\label{eq: projections}
\end{equation}

\begin{notation}
\label{nota: elements in spaces}Given an oriented $d$-lattice $\lat<\RR^{n}$
we denote its homothety class by $\unilat{\lat}\in\latspace{d,n}$,
and its shape by $\shape{\lat}\in\shapespace d$. Given a pair $\brac{\lat,\qlat}$
of lattices in orthogonal subspaces where $\lat$ is primitive of
rank $d$, we denote its homothety class by $\unilat{\brac{\lat,\qlat}}\in\pairspace{d,n}$.
The image of $\lat$ in $\unilatspace d$ (resp.\ of $\qlat$ in
$\unilatspace{n-d}$) is denoted $\unisimlat{\lat}$ (resp.\ $\unisimlat{\qlat}$).

With these notations, Diagram (\ref{eq: projections}) becomes:
\[
\xymatrix{\unilat{\brac{\lat,\qlat}}\ar@{|->}@(ur,ul)[rr]\ar@{|->}[r]\ar@{|->}[d]\ar@{|->}@(dl,ul)[dd]\ar@{|->}@(dr,ul)[ddr] & \unilat{\lat}\ar@{|->}[r]\ar@{|->}[d]\ar@{|->}@(dr,ur)[dd] & V_{\lat}\\
\unisimlat{\qlat}\ar@{|->}[d] & \unisimlat{\lat}\ar@{|->}[d]\\
\shape{\qlat} & \shape{\lat}
}
.
\]
\end{notation}

\begin{rem}
For completeness, we provide an explicit description of the above
maps. The projection $\pairspace{d,n}\to\latspace{d,n}$ is the one
of modding by $\left[\begin{smallmatrix}\idmat d & 0_{d\times n-d}\\
0_{n-d\times d} & \sl{n-d}\left(\RR\right)
\end{smallmatrix}\right]$ \textbackslash{} from the right, and the projection $\latspace{d,n}\to\gras{d,n}$
is modding by $\left[\begin{smallmatrix}\sl d\left(\RR\right) & 0_{d\times n-d}\\
0_{n-d\times d} & \idmat{n-d}
\end{smallmatrix}\right]$ from the right. The projections $\unilatspace i\to\shapespace i$
for $i=d,n-d$ are given by modding from the left by $\so i\left(\RR\right)$.
The projections $\pairspace{d,n}\to\latspace{d,n}$ and $\pairspace{d,n}\to\shapespace{n-d}$
, as mentioned, depend on a choice of coordinates and will be explicated
in Remark \ref{rem: k' detemines projection to U_d and U_n-d}. The
projection $\pairspace{d,n}\to\shapespace d$ is modding from the
left by $\so n\left(\RR\right)\left[\begin{smallmatrix}\idmat d & 0_{d\times n-d}\\
0_{n-d\times d} & P_{n-d}
\end{smallmatrix}\right]$, and similarly for $\pairspace{d,n}\to\shapespace{n-d}$.
\end{rem}

The projections depicted in Diagram (\ref{eq: projections}) are critical
to understanding why the measures we introduced on $\latspace{d,n}$
and $\text{\ensuremath{\pairspace{d,n}}}$ are indeed the natural
uniform measures on these spaces; as we shall see below (Lemma \ref{lem: lift to G''}),
with a consistent choice of normalizations, these measures project
to the uniform measures on $\unilatspace d$, $\shapespace d$ and
$\gras{d,n}$.

\paragraph*{Normalizations of the measures. }

As explained above, all the spaces considered are endowed with invariant
Radon measures that are unique up to a positive scalar, and have finite
volume with respect to these measures. Hence, by fixing the volume
of the entire space, a unique measure is determined. This is what
we now turn to do, where all the unique measures obtained are denoted
by $\vol$ (the space in question should be understood from the context).

In accordance with the fact that $\so d\left(\RR\right)/\so{d-1}\left(\RR\right)$
is diffeomorphic to the sphere $\sphere{d-1}$, we choose a Haar measure
on the special orthogonal group such that $\haar\brac{\so d\left(\RR\right)}=\leb{\sphere{d-1}}\cdot\haar\brac{\so{d-1}\left(\RR\right)}$,
namely:
\[
\haar\brac{\so d\left(\RR\right)}=\prod_{i=1}^{d-1}\leb{\sphere i}.
\]
 We refer to \cite{Volumes} for the fact that, with our standard
choice (Section \ref{subsec: RI Haar measure}) of Haar measure on
$\sl d\left(\RR\right)$, $\vol\brac{\sl d\left(\RR\right)/\sl d\left(\ZZ\right)}=\prod_{i=2}^{d}\zeta\left(i\right)\cdot\prod_{i=1}^{d-1}\leb{\sphere i}/\haar\brac{\so d\left(\RR\right)}$,
where $\z$ is the Riemann Zeta function. Thus, 
\[
\vol\brac{\unilatspace d}=\prod_{i=2}^{d}\zeta\left(i\right).
\]
Similarly, in accordance with the fact that $\so n\left(\RR\right)/\brac{\so d\left(\RR\right)\times\so{n-d}\left(\RR\right)}$
is diffeomorphic to the Grassmannian $\gras{d,n}$, we choose to normalize
the volume of the Grassmannian as $\haar\brac{\so n\left(\RR\right)}/\brac{\haar\brac{\so d\left(\RR\right)}\haar\brac{\so{n-d}\left(\RR\right)}}$,
namely
\[
\vol\brac{\gras{d,n}}=\frac{\prod_{i=1}^{n-1}\leb{\sphere i}}{\prod_{i=1}^{d-1}\leb{\sphere i}\prod_{i=1}^{n-d-1}\leb{\sphere i}}=\frac{\prod_{i=n-d}^{n-1}\leb{\sphere i}}{\prod_{i=1}^{d-1}\leb{\sphere i}}.
\]
Finally, recall from Theorem \ref{thm: Intro} that
\[
\halfK d=\frac{\haar\brac{\so d\left(\RR\right)}}{\left[\so d\left(\RR\right):Z\left(\so d\left(\RR\right)\right)\right]},
\]
where $Z$ denotes the center. Note that the denominator is $1$ if
$d$ is odd and $2$ if it is even. 

As we shall now prove (Lemma \ref{lem: lift to G''}), fixing the
volumes of $\unilatspace d$ and $\gras{d,n}$ determines the volumes
of $\latspace{d,n}$ and $\text{\ensuremath{\pairspace{d,n}}}$ in
correspondence with the identifications in (\ref{eq: identifications}).
In fact, 
\begin{equation}
\begin{aligned}\vol\brac{\unilatspace d}\;=\; & \;\vol\brac{\shapespace d}\halfK d,\\
\vol\brac{\latspace{d,n}}=\; & \;\vol\brac{\unilatspace d}\vol\brac{\gras{d,n}},\\
\vol\brac{\pairspace{d,n}}=\; & \;\vol\brac{\latspace{d,n}}\vol\brac{\unilatspace{n-d}}.
\end{aligned}
\label{eq: product measures}
\end{equation}

\paragraph*{Interactions between the measures. }

We claim that the measures we introduced on $\latspace{d,n}$ and
$\text{\ensuremath{\pairspace{d,n}}}$ are the products of uniform
measures in a manner that corresponds to the identifications in (\ref{eq: identifications}).
For example, if $\latset\subseteq\latspace{d,n}$ projects to the
sets $\groupset\subseteq\unilatspace d$ and $\sphereset\subseteq\gras{d,n}$,
then $\vol\brac{\latset}=\vol\brac{\groupset}\vol\brac{\sphereset}$.
This (and in particular, (\ref{eq: product measures})) is a consequence
of the following lemma.

\begin{lem}
\label{lem: lift to G''}
\begin{enumerate}
\item Assume $\groupset\subseteq\latspace n$ is the lift of $\symset\subseteq\shapespace n$.
If $\symset$ is a BCS then $\groupset$ is, and $\text{\ensuremath{\vol\brac{\groupset}}}=\vol\brac{\symset}\halfK n$.
\item Assume $\latset\subseteq\latspace{d,n}$ is the product to $\groupset\subseteq\unilatspace d$
and $\sphereset\subseteq\gras{d,n}$ (e.g. if $\latset$ is the inverse
image of $\groupset$). If $\groupset$ and $\sphereset$ are BCS's,
then so is $\latset$, and $\vol\brac{\latset}=\vol\brac{\groupset}\vol\brac{\sphereset}$.
\item Assume $\pairset\subseteq\pairspace{d,n}$ is the product of $\groupset\subseteq\unilatspace d$,
$\groupfacset\subseteq\unilatspace{n-d}$ and $\sphereset\subseteq\gras{d,n}$.
If $\groupset$, $\groupfacset$ and $\sphereset$ are BCS's, then
so is $\pairset$, and $\vol\brac{\pairset}=\vol\brac{\groupset}\vol\brac{\groupfacset}\vol\brac{\sphereset}$.
\item Assume $\pairset\subseteq\pairspace{d,n}$ is the product of $\latset\subseteq\latspace{d,n}$
and $\groupfacset\subseteq\unilatspace{n-d}$. If $\latset$ and $\groupfacset$
are BCS's, then so is $\latset$, and $\vol\brac{\pairset}=\vol\brac{\latset}\vol\brac{\groupfacset}$
\end{enumerate}
\end{lem}

\begin{proof}
Part 1 is proved in \cite[Lemma 3.8]{HK_gcd}, Part 3 is a consequence
of Parts 2 and 4, and Parts 2 and 4 are a consequence of \cite[Propositions 6.15 and 6.16]{HK_WellRoundedness}.
\end{proof}

\section{\label{sec: Counting Lattices}Counting lattices: Main theorem}

Theorem (\ref{thm: Intro}) stated in the Introduction is in fact
a consequence of the more general Theorem \ref{thm: A counting},
which we now introduce. It our main result, which we prove throughout
the rest of the paper. This theorem concerns the equidistribution
of $\unilat{\brac{\lat,\factor{\lat}}}\in\pairspace{d,n}$ for primitive
$d$-lattices $\lat$, implying the equidistribution of the projections
 of $\unilat{\brac{\lat,\factor{\lat}}}$ to all the spaces below
$\pairspace{d,n}$, e.g.\ $\brac{\unilat{\lat},\unisimlat{\factor{\lat}}}\in\latspace{d,n}\times\unilatspace{n-d}$,
$\brac{\shape{\lat},\shape{\factor{\lat}}}\in\shapespace d\times\shapespace{n-d}$
, etc. Then Theorem (\ref{thm: Intro}) follows from the fact that
the equidistribution still holds when replacing $\factor{\lat}$ by
its dual, $\perpen{\lat}$, which is the content of Theorem (\ref{thm: counting orthogonal lattices}).

Let us provide more details. First of all, for the counting to hold,
we require that $\unilat{\brac{\lat,\factor{\lat}}}$ and its projections
to all the spaces below it fall in sets that satisfy the following
regularity condition: 
\begin{defn}
\label{def: BCS}A subset $B$ of an orbifold $\manifold$ will be
called \emph{boundary controllable set}, or a BCS, if for every $x\in\mathcal{M}$
there is an open neighborhood $U_{x}$ of $x$ such that $U_{x}\cap\del B$
is contained in a finite union of embedded $C^{1}$ submanifolds of
$\mathcal{M}$, whose dimension is strictly smaller than $\dim\manifold$.
In particular, $B$ is a BCS if its (topological) boundary consists
of finitely many  subsets of embedded $C^{1}$ submanifolds. 
\end{defn}

A few more notions before we can state our main theorem. We refer
to the left component in $\shapespace d\times\shapespace{n-d}$, $\unilatspace d\times\unilatspace{n-d}$
and $\pairspace{d,n}=\cbrac{\unilat{\brac{\lat,\qlat}}}$ as the $d$
component, and to the right component in these spaces as the $n-d$
component. We say that the $d$ (resp.\ $n-d$) component of a set
in one of these spaces is \emph{bounded} if the image of this set
under the projection to $\shapespace d$ (resp.\ $\shapespace{n-d}$)
is bounded. Notice that this is well defined since the projections
to the shape spaces are canonical; they are also proper, so a set
in $\pairspace{d,n}$ (or $\latspace{d,n-d}\times\unilatspace{n-d}$,
or $\unilatspace d\times\unilatspace{n-d}$...) is bounded if and
only if both its $d$ and $n-d$ components are bounded.
\begin{thm}
\label{thm: A counting}Assume that $\sphereset\subseteq\gras{d,n}$,
$\symset\times\factorset\subseteq\shapespace d\times\shapespace{n-d}$,
$\groupset\times\groupfacset\subseteq\unilatspace d\times\unilatspace{n-d}$,
$\latset\subseteq\latspace{d,n}$ and $\pairset\subseteq\pairspace{d,n}$
are BCS's. Then, for every $\e>0$: 
\end{thm}

\begin{enumerate}
\item \label{enu: MainThm_restricted shape}The number of primitive $d$-lattices
$\lat$ of covolume up to $X$ with $V_{\lat}\in\sphereset$ and $\brac{\shape{\lat},\shape{\factor{\lat}}}\in\symset\times\factorset$
is 
\[
\frac{\vol\brac{\symset}\vol\brac{\factorset}\vol\brac{\sphereset}}{n\cdot\prod_{i=2}^{n}\zeta\left(i\right)}\halfK d\halfK{n-d}\cdot X^{n}+\text{error term}.
\]
\item \label{enu: MainThm_restricted G''}The number of primitive $d$-lattices
$\lat$ of covolume up to $X$ with $V_{\lat}\in\sphereset$ and $\brac{\unisimlat{\lat},\unisimlat{\factor{\lat}}}\in\groupset\times\groupfacset$
is 
\[
\frac{\vol\brac{\groupset}\vol\brac{\groupfacset}\vol\brac{\sphereset}}{n\cdot\prod_{i=2}^{n}\zeta\left(i\right)}\cdot X^{n}+\text{error term}.
\]
\item \label{enu: MainThm_restricted components of pairs}The number of
primitive $d$-lattices $\lat$ of covolume up to $X$ with $\unilat{\lat}\in\latset$
and $\unisimlat{\factor{\lat}}\in\groupfacset$ is 
\[
\frac{\vol\brac{\latset}\vol\brac{\groupfacset}}{n\cdot\prod_{i=2}^{n}\zeta\left(i\right)}\cdot X^{n}+\text{error term}.
\]
In particular, the number of primitive $\lat$ with $\covol{\lat}\leq X$
whose homothety class  lies inside a BCS $\latset\subseteq\latspace{d,n}$
is $\frac{\vol\brac{\latset}}{\prod_{i=n-d+1}^{n}\zeta\left(i\right)}\cdot\frac{X^{n}}{n}$,
plus an error term. 
\item \label{enu: MainThm_restricted pairs}The number of primitive $d$-lattices
$\lat$ of covolume up to $X$ with $\unilat{\brac{\lat,\factor{\lat}}}\in\pairset$
is
\[
\frac{\vol\brac{\pairset}}{n\cdot\prod_{i=2}^{n}\zeta\left(i\right)}\cdot X^{n}+\text{error term}.
\]
\end{enumerate}
The error term is $O_{\e}\brac{X^{n-\frac{\errexp_{n}n}{2\left(n-d-1\right)\left(n^{2}-1\right)+n^{2}}+\e}}$
with $\errexp_{n}=\left\lceil \left(n-1\right)/2\right\rceil /4n^{2}$
when the $d$ component is bounded, $O_{\e}\brac{X^{n-\frac{\errexp_{n}n}{2\left(d-1\right)\left(n^{2}-1\right)+n^{2}}+\e}}$
when the $n-d$ component is bounded, $O_{\e}\brac{X^{n-\frac{\errexp_{n}n}{2\left(\max\left\{ d,n-d\right\} -1\right)\left(n^{2}-1\right)+n^{2}}+\e}}$
when neither of them is bounded and $O_{\e}\brac{X^{n-n\errexp_{n}+\e}}$
when both of them are bounded. 

\begin{rem}
\begin{enumerate}
\item Regarding the aforementioned work of Schmidt: our leading constants
in Theorem \ref{thm: A counting} coincide with Schmidt's constant
$c_{d,n}$ (\cite[Thm 2]{Schmidt_98}, \cite[Thm 1.2]{Schmidt_15})
when the measures on the appearing subspaces are normalized to probability
measures. In fact, as can be computed from the chosen normalizations
of the volumes of the spaces that appear in Theorem \ref{thm: A counting},
along with the fact that $\leb{\sphere i}=\left(i+1\right)\leb{\ball{i+1}}$,
the relation between Schmidt's constant and our chosen normalizations
is given by
\[
c_{d,n}=\frac{\vol\brac{\gras{d,n}}}{2}\cdot\frac{\vol\brac{\unilatspace d}\vol\brac{\unilatspace{n-d}}}{\vol\brac{\unilatspace n}}.
\]
The $\frac{1}{2}$ factor is due to the fact that the lattices we
count carry an orientation, so every non-oriented lattice is counted
twice. We also note that a set that is BCS is Jordan measurable, and
indeed Schmidt (in \cite{Schmidt_98}) provides an example for how
the asymptotic formula for number of $d$-lattices with shapes in
$\symset$ fails when $\symset$ is not Jordan measurable. 
\item The case of $d=n-1$ was also obtained in \cite{Marklof_10}, using
a dynamical approach, as well as in \cite{HK_gcd}. It was also considered
in \cite{AES_16A,AES_16B,EMSS_16,ERW17,B19} in a more delicate setting.
\item Primitive $d$-lattices are in one-to-one correspondence with rational
subspaces in $\RR^{n}$. These spaces are the rational points on the
Grassmanian variety: the projective variety consisting of all the
$d$\textendash dimensional spaces in $\RR^{n}$. Therefore, the aforementioned
result of Schmidt can be read as the counting of rational points up
to a bounded height in the Grassmannian variety (the height being
the covolume of the unique primitive lattice in the space). As such,
it provides yet another example where the Manin conjecture \cite{FMT89,P95}
on counting rational points in varieties holds. The more refined counting
we suggest in Theorem \ref{thm: A counting} (Part \ref{enu: MainThm_restricted pairs})
plays a key role in the intensive study of rational points on the
Grassmannian conducted in \cite{BHW_grassmannians}. In this paper,
the authors confirm a modification to the Manin conjecture suggested
by Peyre \cite{P17}, and obtain an equidistribution result for the
integral lattices in the rational tangent bundle of the Grassmannian. 
\end{enumerate}
\end{rem}

Note that Theorem (\ref{thm: Intro}) is just Part \ref{enu: MainThm_restricted shape}
of Theorem \ref{thm: A counting}, only with $\perpen{\lat}$ instead
of $\factor{\lat}$. As we have already mentioned, these lattices
are dual to one another; we denote the dual of a lattice $\lat$ by
$\dual{\lat}$. Then Theorem (\ref{thm: Intro}) is obtained from
Theorem \ref{thm: A counting} and the following.

\begin{thm}
\label{thm: counting orthogonal lattices}A version of Theorem \ref{thm: A counting}
holds when replacing $\left(\lat,\factor{\lat}\right)$ by each of
$\left(\lat,\perpen{\lat}\right)$, $\left(\dual{\lat},\perpen{\lat}\right)$,
$\left(\dual{\lat},\factor{\lat}\right)$. 
\end{thm}

Indeed, note that in all of the pairs above, the right-hand lattice
spans the orthogonal subspace to the one spanned by the left-hand
lattice. In other words, their homothety classes are elements in $\pairspace{d,n}$. 

\begin{proof}[Proof of Theorem \ref{thm: counting orthogonal lattices}]
Note that the three pairs $\left(\lat,\perpen{\lat}\right)$, $\left(\dual{\lat},\perpen{\lat}\right)$
and $\left(\dual{\lat},\factor{\lat}\right)$ are obtained from the
pair in Theorem \ref{thm: A counting} $\left(\lat,\factor{\lat}\right)$
by taking the dual of one of $\lat,\factor{\lat}$ or of both. According
to Proposition \ref{prop: All that Dual... and P_d,n} and Proposition
\ref{rem: dual of unimodular}, taking the dual of one of the lattices
in a pair or of both is a measure preserving auto-diffeomorphism of
$\pairspace{d,n}$. 
\end{proof}

\section{\label{sec: RI coordinates}Refined Iwasawa components in $\protect\sl n\left(\protect\RR\right)$.}

\subsection{\label{subsec: Introducing RI coordinates}Refining the Iwasawa decomposition
of $\protect\sl n\left(\protect\RR\right)$}

Set $G=G_{n}:=\sl n\left(\RR\right)$ and let $G=KAN$ be the Iwasawa
decomposition of $G$, meaning that $K=K_{n}$ is $\so n\left(\RR\right)$,
$A=A_{n}$ is the diagonal subgroup in $G$, and $N=N_{n}$ is the
subgroup of upper unipotent matrices. We also let $P_{n}=A_{n}N_{n}$.
Consider the following isomorphic copy of $\sl d\left(\RR\right)\times\sl{n-d}\left(\RR\right)$
inside $G$,
\[
G^{\dprime}:=\left[\begin{array}{c|c}
G_{d} & 0_{d,n-d}\\
\hline 0_{n-d,d} & G_{n-d}
\end{array}\right]=\left[\begin{array}{c|c}
\sl d\left(\RR\right) & 0_{d,n-d}\\
\hline 0_{n-d,d} & \sl{n-d}\left(\RR\right)
\end{array}\right].
\]
Write $G^{\dprime}=K^{\dprime}A^{\dprime}N^{\dprime}$ for the Iwasawa
decomposition of $G^{\dprime}$, namely
\begin{eqnarray*}
K^{\dprime} & := & K\cap G^{\dprime}=\left[\begin{array}{c|c}
K_{d} & 0_{d,n-d}\\
\hline 0_{n-d,d} & K_{n-d}
\end{array}\right]=\left[\begin{array}{c|c}
\so d\left(\RR\right) & 0_{d,n-d}\\
\hline 0_{n-d,d} & \so{n-d}\left(\RR\right)
\end{array}\right],\\
A^{\dprime} & := & A\cap G^{\dprime}=\left[\begin{array}{c|c}
A_{d} & 0_{d,n-d}\\
\hline 0_{n-d,d} & A_{n-d}
\end{array}\right],\\
N^{\dprime} & := & N\cap G^{\dprime}=\left[\begin{array}{c|c}
N_{d} & 0_{d,n-d}\\
\hline 0_{n-d,d} & N_{n-d}
\end{array}\right].
\end{eqnarray*}
Let 
\[
P^{\dprime}:=A^{\dprime}N^{\dprime}\;\text{and}\;\wc:=KP^{\dprime}
\]
(note that $\wc$ is not a group, but it is a smooth manifold). To
complete the definition of the Refined Iwasawa decomposition, we define
$K^{\prime},A^{\prime},N^{\prime}$ that complete $K^{\dprime},A^{\dprime},N^{\dprime}$
to $K$, $A$ and $N$ respectively. Let
\[
N^{\prime}:=\left[\begin{array}{c|c}
\idmat d & \RR^{d,n-d}\\
\hline 0_{n-d,d} & \idmat{n-d}
\end{array}\right],\quad A^{\prime}:=\left[\begin{array}{c|c}
a^{\frac{1}{d}}\idmat d & 0_{d,n-d}\\
\hline 0_{n-d,d} & a^{-\frac{1}{n-d}}\idmat{n-d}
\end{array}\right],
\]
and observe that $N=N^{\dprime}N^{\prime}$, $A=A^{\dprime}A^{\prime}$,
and that $A^{\prime}$ is a one-parameter subgroup of $A$ which commutes
with $G^{\dprime}$. Often we would like to restrict to the upper
$d\times d$ (resp.\ lower $\left(n-d\right)\times\left(n-d\right)$)
block of $G^{\dprime}$, hence we denote 
\[
G_{d}^{\dprime}:=\left[\begin{array}{c|c}
G_{d} & 0_{d,n-d}\\
\hline 0_{n-d,d} & \idmat{n-d}
\end{array}\right]=\left[\begin{array}{c|c}
\sl d\left(\RR\right) & 0_{d,n-d}\\
\hline 0_{n-d,d} & \idmat{n-d}
\end{array}\right]
\]
(resp. $G_{n-d}^{\dprime}:=\left[\begin{array}{c|c}
\idmat d & 0_{d,n-d}\\
\hline 0_{n-d,d} & G_{n-d}
\end{array}\right]$). Similarly for $K^{\dprime}$,$P^{\dprime}$ and $A^{\dprime}$. 

Fix a transversal $K^{\prime}$ of the diffeomorphism $K/K^{\dprime}\to\gras{d,n}$,
meaning that $K=K^{\prime}K^{\dprime}$; later (Proposition \ref{prop: spread models that we need})
we will restrict to a specific choice of a transversal, with some
desirable properties. Then $\wc$ is also $K^{\prime}G^{\dprime}$,
and we let 
\[
\wc_{d}:=K^{\prime}G_{d}^{\dprime}.
\]
Note that $\wc=\wc_{d}G_{n-d}^{\dprime}$. Then the Refined Iwasawa
(or RI, for short) decomposition is given by 

\[
G=K^{\prime}K^{\dprime}A^{\dprime}A^{\prime}N^{\dprime}N^{\prime}=K^{\prime}G^{\dprime}A^{\prime}N^{\prime}=KP^{\dprime}A^{\prime}N^{\prime}=\wc A^{\prime}N^{\prime}=\wc_{d}G_{n-d}^{\dprime}A^{\prime}N^{\prime}.
\]

\begin{rem}
\label{rem: true quotients L_d,n and P_d,n}In particular, the two
quotients with whom we expressed $\latspace{d,n}$ and $\pairspace{d,n}$
in the Introduction indeed agree. For $\latspace{d,n}$, we have that
$\sl n\left(\RR\right)=KP^{\dprime}A^{\prime}N^{\prime}=KP_{d}^{\dprime}P_{n-d}^{\dprime}A^{\prime}N^{\prime}$,
so 
\[
\latspace{d,n}=\sl n\left(\RR\right)/G_{d}^{\dprime}\left(\ZZ\right)G_{n-d}^{\dprime}A^{\prime}N^{\prime}=\sl n\left(\RR\right)/G_{d}^{\dprime}\left(\ZZ\right)K_{n-d}^{\dprime}P_{n-d}^{\dprime}A^{\prime}N^{\prime}\diffeo KP_{d}^{\dprime}/G_{d}^{\dprime}\left(\ZZ\right)K_{n-d}^{\dprime}.
\]
Similarly,
\[
\pairspace{d,n}=\sl n\left(\RR\right)/G^{\dprime}\left(\ZZ\right)N^{\prime}A^{\prime}\diffeo KP^{\dprime}/G^{\dprime}\left(\ZZ\right).
\]
\end{rem}

\begin{rem}
\label{rem: k' detemines projection to U_d and U_n-d}The choice of
a transversal $K^{\prime}$ determines projections $\latspace{d,n}\to\unilatspace d$
and $\pairspace{d,n}\to\unilatspace{n-d}$: both these maps are induced
by 
\begin{eqnarray*}
\sl n(\RR) & \to & \sl n(\RR)\\
k'g^{\dprime}a^{\prime}n^{\prime} & \mapsto & g^{\dprime}a^{\prime}n^{\prime}.
\end{eqnarray*}
In particular, for every $d$ lattice $\lat$ and every $\left(n-d\right)$
lattice $\qlat$ in $\perpen{V_{\lat}}$, the projections $\unisimlat{\lat}\in\unilatspace d$
and $\unisimlat{\qlat}\in\unilatspace{n-d}$ from Notation \ref{nota: elements in spaces}
are now determined.
\end{rem}

\subsection{Parameterizations of the RI components}

Our use of the RI decomposition of $\sl n\left(\RR\right)$ is motivated
by the fact that, as we shall see in Section \ref{sec: Fundamental domains},
the components appearing in this decomposition or some special subsets
of theirs are in a sense isomorphic to the homogeneous spaces of $\sl n\left(\RR\right)$
that we mentioned in the Introduction. Specifically, we shall state
(Proposition \ref{prop: spread models that we need}) that certain
subsets of $G_{d}$, $P_{d}$, $\wc_{d}$ and $\wc$ are  parameterized
by the spaces $\unilatspace d$, $\shapespace d$, $\latspace{d,n}$
and $\pairspace{d,n}$; in particular, subsetes of $G^{\dprime}=\left[\begin{smallmatrix}G_{d} & 0\\
0 & G_{n-d}
\end{smallmatrix}\right]$ and $P^{\dprime}=\left[\begin{smallmatrix}P_{d} & 0\\
0 & P_{n-d}
\end{smallmatrix}\right]$ are parameterized by $\unilatspace d\times\unilatspace{n-d}$ and
$\shapespace d\times\shapespace{n-d}$ respectively. Also, it is clear
that $K^{\prime}$ and $N^{\prime}$ can be parameterized $\gras{d,n}$
and $\RR^{d\left(n-d\right)}$; the parameterization for $K^{\prime}\subset K$
by the Grassmannian will be the inverse of the restriction of the
quotient map, and a choice of an isomorphism for $N^{\prime}$ is
arbitrary (but fixed). We let $k_{U}^{\prime}$ denote the element
of $K^{\prime}$ that corresponds to an oriented $d$-dimensional
subspace $U$ in $\gras{d,n}$. 

When $\domN$ is a subset of a space that is parameterized by a RI
component $\GIcomp$ (or a subset of it), we let $\parby{\GIcomp}{\domN}$
denote the image of $\domN$ in $\GIcomp$ under this parameterization.
For example, $\domN\subset\RR^{d\left(n-d\right)}$, then $N_{\domN}^{\prime}$
denotes its image in $N^{\prime}$. Finally, the groups $A,A^{\dprime},A^{\prime}$
are clearly isomorphic to $\RR^{n-1}$, $\RR^{d-1}\times\RR^{n-d-1}$
and $\RR$ respectively. We choose the following parameterizations:
$a=\diag{a_{1},\ldots a_{n}}$ will be written as $a_{\underline{h}}=a_{\left(h_{1},\ldots,h_{n-1}\right)}$
if 
\[
\left(a_{1},a_{2},\ldots,a_{n-1},a_{n}\right)=\brac{e^{-h_{1}/2},e^{\left(h_{1}-h_{2}\right)/2},\ldots,e^{\left(h_{n-2}-h_{n-1}\right)/2},e^{h_{n-1}/2}}.
\]
Accordingly, we write $A^{\dprime}\ni a^{\dprime}=a_{\underline{s},\underline{w}}^{\dprime}$
with $\underline{s}=\left(s_{1},\ldots,s_{d-1}\right)\in\RR^{d-1}$
and $\underline{w}=\left(w_{1},\ldots,w_{n-d-1}\right)\in\RR^{n-d-1}$
if $a^{\dprime}=\left[\begin{smallmatrix}a_{\underline{s}}\in A_{d}\\
 & a_{\underline{w}}\in A_{n-d}
\end{smallmatrix}\right]$. For an element in $A^{\prime}$, we let 
\[
a_{t}^{\prime}:=\left[\begin{smallmatrix}e^{\frac{t}{d}}\idmat d & 0\\
0 & e^{-\frac{t}{n-d}}\idmat{n-d}
\end{smallmatrix}\right].
\]

\subsection{\label{subsec: RI Haar measure}RI decomposition of the Haar measure
on $\protect\sl n\left(\protect\RR\right)$}

It is well known (e.g. \cite[Prop. 8.43]{Knapp}) that a Haar measure
on $\sl n\left(\RR\right)$ can be decomposed according to the Iwasawa
components of $\sl n\left(\RR\right)$. Let us extend this to a Refined
Iwasawa decomposition of the Haar measure on $\sl n\left(\RR\right)$.

For every $\GIcomp\subset G=\sl n\left(\RR\right)$ appearing as a
component in the Iwasawa or Refined Iwasawa decompositions of $G$,
we let $\mu_{\GIcomp}$ denote a measure on $\GIcomp$ as follows:
$\mu_{K},\mu_{N}$ are Haar measures, and so do $\mu_{K^{\dprime}}$,
$\mu_{N^{\dprime}}$, $\mu_{P^{\dprime}}$, $\mu_{G^{\dprime}},\mu_{G},\mu_{G_{i}},\mu_{K_{i}},\mu_{N_{i}}$
and $\mu_{N^{\prime}}$ for $i=d,n-d$. The measures $\mu_{N}$, $\mu_{N^{\prime}},\mu_{N_{i}}$
and $\mu_{N^{\dprime}}$ are Lebesgue; as $N=N^{\dprime}\ltimes N^{\prime}$
and all three groups are unimodular, $\mu_{N}=\mu_{N^{\dprime}}\times\mu_{N^{\prime}}$.
We assume that the Haar measures $\mu_{K}$ and $\mu_{K^{\dprime}}$
are normalized such that $\mu_{K^{\dprime}}\left(K^{\dprime}\right)=\mu_{K_{d}}\left(K_{d}\right)\mu_{K_{n-d}}\left(K_{n-d}\right)$,
where (as we have mentioned in Section \ref{sec: Introduction}) $\mu_{K_{i}}\left(K_{i}\right)=\prod_{j=1}^{i-1}\leb{\sphere j}$.
Since $K^{\prime}$ parameterizes $\gras{d,n}=K/K^{\dprime}$, we
can endow it with a measure $\mu_{K^{\prime}}$ that is the pullback
of a $K$-invariant Radon measure on $K/K^{\dprime}$ so that $\mu_{K}=\mu_{K^{\prime}}\times\mu_{K^{\dprime}}$. 

The measures $\mu_{A},\mu_{A^{\prime}},\mu_{A^{\dprime}}$ are Radon
measures such that 
\[
\mu_{A^{\prime}}=e^{nt}dt,\quad\mu_{A_{i}}=\prod_{j=1}^{i}e^{-h_{j}}dh_{j}
\]
 and 
\[
\mu_{A^{\dprime}}=\mu_{A_{d}^{\dprime}}\times\mu_{A_{n-d}^{\dprime}},\quad\mu_{A}=\mu_{A^{\prime}}\times\mu_{A^{\dprime}}.
\]
Note that these measures are non-Haar. We use the measures on $A_{i},A^{\dprime}$
to determine a (normalization of the) Haar measure on $G_{i}$:
\[
\mu_{G_{i}}=\mu_{K_{i}}\times\mu_{A_{i}}\times\mu_{N_{i}};\quad\mu_{G^{\dprime}}=\mu_{G_{d}^{\dprime}}\times\mu_{G_{n-d}^{\dprime}}.
\]
Since $\wc$ is diffeomorphic to the group $K\times P^{\dprime}$,
we endow it with the Haar measure on this group: $\mu_{\wc}=\mu_{K}\times\mu_{P^{\dprime}}$.
Since $\mu_{K}=\mu_{K^{\prime}}\times\mu_{K^{\dprime}}$, we also
have that also $\mu_{\wc}=\mu_{K^{\prime}}\times\mu_{G^{\dprime}}$.
Since the product map $K^{\prime}\times G_{d}^{\dprime}\to Q_{d}=K^{\prime}G_{d}^{\dprime}$
 is a Borel bijection, we may define $\mu_{Q_{d}}=\mu_{K^{\prime}}\times\mu_{G_{d}^{\dprime}}$. 

All in all we have the following:
\begin{align}
\begin{array}{c}
\mu_{G}=\mu_{K^{\prime}}\times\mu_{K^{\dprime}}\times\mu_{A^{\dprime}}\times\mu_{A^{\prime}}\times\mu_{N^{\dprime}}\times\mu_{N^{\prime}}=\mu_{K^{\prime}}\times\mu_{G^{\dprime}}\times\mu_{A^{\prime}}\times\mu_{N^{\prime}}\\
=\mu_{\wc}\times\mu_{A^{\prime}}\times\mu_{N^{\prime}}=\mu_{\wc_{d}}\times\mu_{G_{n-d}^{\dprime}}\times\mu_{A^{\dprime}}\times\mu_{N^{\dprime}},\\
\\
\mu_{G^{\dprime}}=\mu_{G_{d}^{\dprime}}\times\mu_{G_{n-d}^{\dprime}},\:\mu_{P^{\dprime}}=\mu_{P_{d}^{\dprime}}\times\mu_{P_{n-d}^{\dprime}},\:\mu_{K^{\dprime}}=\mu_{K_{d}^{\dprime}}\times\mu_{K_{n-d}^{\dprime}},\:\mu_{A^{\dprime}}=\mu_{A_{d}^{\dprime}}\times\mu_{A_{n-d}^{\dprime}}.
\end{array}\label{eq: Haar measure on G}
\end{align}
To abbreviate, we will use  $\mu$ instead of $\mu_{G}$. 

\subsection{Interpretation of the RI components of $g\in\protect\sl n\left(\protect\RR\right)$}

Like the Iwasawa decomposition, the Refined Iwasawa decomposition
induces coordinates on $\sl n\left(\RR\right)$. The RI components
of an element $g\in\sl n\left(\RR\right)$ encode certain information
regarding the lattices spanned by its columns, as explained in the
proposition below. To state it, we extend the definition of primitiveness
and of factor lattices from $d$-lattices in $\ZZ^{n}$, to $d$-lattices
in any full lattice of $\RR^{n}$. 
\begin{defn}
Assume that a $d$-lattice $\lat$ is contained inside a full lattice
$\latfull<\RR^{n}$. We say that $\lat$ is \emph{primitive} inside
(or with respect to) $\latfull$ if $\lat=\latfull\cap V_{\lat}$.
In other words, if there is no subgroup of $\latfull$ that lies inside
$V_{\lat}$ and properly contains $\lat$. Given a $d$-lattice $\lat$
that is primitive inside $\latfull$, we define the \emph{factor
lattice of $\lat$ (w.r.t.\ $\latfull$)}, denoted $\factorg{\lat}{\latfull}$,
as the orthogonal projection of $\latfull$ into the space $\perpen{\brac{V_{\lat}}}$.
When $\lat$ is primitive inside $\ZZ^{n}$, we omit the explicit
mentioning of $\ZZ^{n}$, and say just that $\lat$ is primitive.
Accordingly, we denote $\factor{\lat}$, and say that it is the factor
lattice of $\lat$.
\end{defn}

Let us also introduce the following notation. 
\begin{notation}
For $g\in\sl n\left(\RR\right)$ and $d\in\left\{ 1,\ldots,n-1\right\} $,
let $\lat_{g}$ denote the lattice spanned by the columns of $g$,
and let $\lat_{g}^{d}$ denote the lattice spanned by the first $d$
columns of $g$. Let $\latlast{\lat_{g}}d$ denote the lattice spanned
by the last $d$ columns of $\lat_{g}$.  Finally, given a lattice
$\lat$ in $\RR^{n}$, recall that $V_{\lat}$ denotes the linear
space spanned by $\lat$. 
\end{notation}

\begin{prop}
\label{prop: explicit RI coordinates of g}Let $g\in\sl n\left(\RR\right)$
and $1\leq d\leq n-1$. Denote $\lat=\lat_{g}^{d}$, the lattice spanned
by the first $d$ columns of $g$, and $\lat^{\sharp}=\factorg{\lat}{\lat_{g}}$,
the factor lattice of $\lat$ w.r.t.\ $\lat_{g}$. Write $g=kan=qa_{t}^{\prime}n^{\prime}$
with $q=k_{U}^{\prime}g^{\dprime}=k_{U}^{\prime}k^{\dprime}a_{\underline{s},\wvec}^{\dprime}n^{\dprime}$.
Let $g^{\dprime}=\left[\begin{smallmatrix}g_{d} & 0\\
0 & g_{n-d}
\end{smallmatrix}\right]$, $g_{d}^{\dprime}=\left[\begin{smallmatrix}g_{d} & 0\\
0 & \idmat{n-d}
\end{smallmatrix}\right]$, $g_{n-d}^{\dprime}=\left[\begin{smallmatrix}\idmat d & 0\\
0 & g_{n-d}
\end{smallmatrix}\right]$, and similarly for $p^{\dprime}$.  The RI components of $g$ represent
parameters related to $\lat$ in the following way:
\[
\begin{array}{cccccccccc}
\brac i & U & = & V_{\lat} &  & \brac i^{\sharp} & \perpen U & = & V_{\lat^{\sharp}}\\
\brac{ii} & e^{t} & = & \covol{\lat} &  & \brac{ii}^{\sharp} & e^{-t} & = & \covol{\lat^{\sharp}}\\
\brac{iii} & e^{\frac{it}{d}-\frac{s_{i}}{2}} & = & \covol{\lat^{i}} &  & \brac{iii}^{\sharp} & e^{-\frac{jt}{n-d}-\frac{w_{j}}{2}} & = & \covol{\left(\lat^{\sharp}\right)^{j}}\\
\brac{iv} & \lat_{q}^{d} & \in & \unilat{\lat} &  & \brac{iv}^{\sharp} & \latlast{\lat_{q}}{n-d} & \in & \unilat{\lat^{\sharp}}\\
\brac v & \lat_{g_{d}} & \in & \unisimlat{\lat} &  & \brac v^{\sharp} & \lat_{g_{n-d}} & \in & \unisimlat{\lat^{\sharp}}\\
\brac{vi} & \ensuremath{\lat_{p_{d}}} & \in & \shape{\lat} &  & \brac{vi}^{\sharp} & \ensuremath{\lat_{p_{n-d}}} & \in & \shape{\lat^{\sharp}}
\end{array}
\]
for every $1\leq i\leq d$ and $1\leq j\leq n-d$, and: $\brac{vii}$
$\brac{\lat_{q}^{d},\latlast{\lat_{q}}{n-d}}\in\unilat{\brac{\lat_{g}^{d},\factorg{\brac{\lat_{g}^{d}}}{\lat_{g}}}}$.
\end{prop}

\begin{proof}
Since the columns of $k$ are obtained by performing the Gram-Schmidt
orthogonalization procedure on the columns of $g$, we have that 
the first $d$ columns of $k$ span $V_{\lat}$. Since $k^{\prime}=k\left(k^{\dprime}\right)^{-1}$
where $k^{\dprime}\in\left[\begin{smallmatrix}\so d\left(\RR\right) & 0\\
0 & \so{n-d}\left(\RR\right)
\end{smallmatrix}\right]$, then the first $d$ columns of $k^{\prime}$ span (and are in fact
an orthonormal basis to) the same space as the first $d$ columns
of $k$, which is $V_{\lat}$. This proves $\brac i$ and  $\brac i^{\sharp}$,
by definition of orientation on $\perpen{V_{\lat}}$. 

Write $g\brac{n^{\prime}}^{-1}\brac{a^{\prime}}^{-1}=k^{\prime}g^{\dprime}=q$;
right multiplication by an element of $N^{\prime}$ does not change
the first $d$ columns of $g$, and right multiplication by $\brac{a^{\prime}}^{-1}$
multiplies each of these columns by $e^{-\frac{t}{d}}$. This means
that $e^{-\frac{t}{d}}\lat=\lat_{q}^{d}$, proving $\left(iv\right)$,
and $\left(v\right)$, $\left(vi\right)$ directly follow. Also, notice
that $\lat_{q}^{d}$ has covolume one, since it is a rotation of
the unimodular $\lat_{g_{d}^{\dprime}}^{d}$ ; then, by considering
the covolumes of the lattices on both sides, we obtain $e^{-t}\covol{\lat}=1$,
proving $\left(ii\right)$. 

Considering $g^{\dprime}$, it is clear that the lattice $\lat_{g_{n-d}}$
is the factor lattice of $\lat_{g_{d}}$ w.r.t.\ $\lat_{g^{\dprime}}$.
Rotating it by left multiplication by $k^{\prime}$, we have that
the lattice  $\latlast{\lat_{q}}{n-d}$ is the factor lattice of
$\lat_{q}^{d}$ w.r.t.\ $\lat_{q}$. But since $q=g\brac{n^{\prime}}^{-1}\brac{a^{\prime}}^{-1}$,
we may also say that $\latlast{\lat_{q}}{n-d}$ is the factor lattice
of $\lat_{g\brac{n^{\prime}}^{-1}\brac{a^{\prime}}^{-1}}^{d}<V_{\lat}$
w.r.t.\ $\lat_{g\brac{n^{\prime}}^{-1}\brac{a^{\prime}}^{-1}}$,
namely it is the projection of $\lat_{g\brac{n^{\prime}}^{-1}\brac{a^{\prime}}^{-1}}$
to $\perpen{V_{\lat}}$. Noticing that this projection kills the contribution
of $\brac{n^{\prime}}^{-1}$, as well as the first $d$ columns of
$g\brac{n^{\prime}}^{-1}\brac{a^{\prime}}^{-1}$, we remain only with
the projection of the lattice spanned by the last $n-d$ columns of
$g$, on which $\brac{a^{\prime}}^{-1}$ acts as multiplication by
$e^{-\frac{t}{n-d}}$. In other words, $\latlast{\lat_{q}}{n-d}$
is in fact  $e^{-\frac{t}{n-d}}\lat^{\sharp}$. This proves $\left(iv\right)^{\sharp}$,
from which $\left(v\right)^{\sharp}$ and $\left(vi\right)^{\sharp}$
follow, and then similarly to how we proved $\left(ii\right)$ we
also obtain $\left(ii\right)^{\sharp}$\footnote{The fact that $\covol{\lat^{\sharp}}=\covol{\lat_{g}}/\covol{\lat}$
is also proved in the Appendix, Proposition \ref{prop: factor covolume}.}. Since $\lat_{q}^{d}$ and $\latlast{\lat_{q}}{n-d}$ span orthogonal
subspaces and are both unimodular, then $\left(iv\right)$ and $\left(iv\right)^{\sharp}$
imply $\brac{vii}$.

It is well known that if $g=kan$ and $a=\diag{\a_{1},\ldots,\a_{r}}$,
then ${\scriptstyle \prod_{1}^{i}}\a_{j}=\covol{\lat_{g}^{i}}$. Since
the lattice $\lat_{g_{d}^{\dprime}}^{d}$ is a rotation of $e^{-t/d}\lat$,
it has the same covolume and partial covolumes; writing $g_{d}^{\dprime}=k_{d}^{\dprime}a_{d}^{\dprime}n_{d}^{\dprime}$,
we have that the product of the first $1\leq i\leq d$ entries of
$a_{d}^{\dprime}$ is $e^{-\frac{it}{d}}\covol{\lat_{g_{d}^{\dprime}}^{i}}=e^{-\frac{it}{d}}\covol{\lat_{g}^{i}}$.
On the other hand, since we have $a_{d}^{\dprime}=\diag{e^{-\frac{s_{1}}{2}},e^{\frac{s_{1}-s_{2}}{2}},\ldots,e^{\frac{s_{d-2}-s_{d-1}}{2}},e^{\frac{s_{d-1}}{2}},1,\ldots,1}$,
comparing the products of the first $i$ elements implies $\covol{\lat^{i}}\cdot e^{-\frac{it}{d}}=e^{-\frac{s_{i}}{2}}$
and proves $\left(iii\right)$. Doing the same for $a_{n-d}^{\dprime}=\diag{1,\ldots,1,e^{-\frac{w_{1}}{2}},e^{\frac{w_{1}-w_{2}}{2}},\ldots,e^{\frac{w_{n-d-2}-w_{n-d-1}}{2}},e^{\frac{w_{n-d-1}}{2}}}$,
while recalling that $\latlast{\lat_{g_{n-d}^{\dprime}}}{n-d}$ is
a rotation of $e^{\frac{t}{n-d}}\lat^{\sharp}$, implies $e^{\frac{jt}{n-d}}\covol{\brac{\lat^{\sharp}}^{j}}=e^{-\frac{w_{j}}{2}}$
for every $1\leq j\leq n-d$ and proves $\left(iii\right)^{\sharp}$. 
\end{proof}

\section{\label{sec: Fundamental domains}Sets of representatives for the
spaces in Theorem \ref{thm: A counting}}

In this section we find ``isomorphic'' copies of the spaces $\shapespace d\times\shapespace{n-d}$,
$\unilatspace d\times\unilatspace{n-d}$, $\latspace{d,n}$, $\pairspace{d,n}$
and $\gras{d,n}$ inside the RI components of $\sl n\left(\RR\right)$,
namely subsets of the RI components that are parameterized by these
spaces. These will be in fact sets of representatives for the cosets
that are the elements of these spaces, that are carefully chosen so
that they imitate the geometry of the spaces as much as possible;
for example, they will have the property that under the quotient map
(that is one to one on a set of representatives), a BCS in the space
will correspond to a BCS in the set. In \cite[Section 6]{HK_WellRoundedness}
we have defined such ``good'' sets of representatives, and named
them \emph{spread models}. For the convenience of the reader, we omit
the definition, and instead list here the properties of spread models
that will be used in this work. 
\begin{prop}[{see \cite[Propositions 6.5, 6.7 and 6.10]{HK_WellRoundedness}}]
\label{prop: spread models - properties}Suppose that a Lie group
$H$ acts properly and smoothly on a manifold $\manifold$ and $F\subset\manifold$
is a spread model for the action of $H$. For measurable sets $B\subseteq\manifold/H$
and $\mathcal{B}\subseteq H$ the following holds:
\begin{enumerate}
\item If $B$ and $\mathcal{B}$ are BCS, then so does $\manifold_{B}\cdot\mathcal{B}\subseteq\manifold$,
where $\manifold_{B}$ is a full set of representatives for $B$ inside
$F$; in the case where $H$ is also discrete, then $B$ is BCS iff
$\manifold_{B}$ is BCS.
\item if $\mu_{H}$ and $\mu_{\manifold}$ are Radon measures on their respective
spaces, then there is a unique Radon measure $\mu_{\manifold/H}$
on $\manifold/H$ so that 
\[
\mu_{\manifold}\left(\manifold_{B}\cdot\mathcal{B}\right)=\mu_{\manifold/H}\left(B\right)\cdot\mu_{H}\left(\mathcal{B}\right).
\]
 
\end{enumerate}
\end{prop}

The spread model we use for the space of shapes $\shapespace i$ ($i=d,n-d$)
is a fundamental domain for the action of $\sl i\left(\ZZ\right)$
on $\so i\left(\RR\right)\backslash\sl i\left(\RR\right)\cong P_{i}$,
whose construction is essentially due to Siegel and is made explicit
in \cite[Section 7]{Grenier_93,Schmidt_98,HK_WellRoundedness}. In
the case of $i=2$, where $P_{2}$ is diffeomorphic with the hyperbolic
upper half plane, this fundamental domain is the well known set depicted
in Figure \ref{fig: fund dom SL(2,Z)}.

\begin{figure}
\begin{centering}
\includegraphics[scale=0.4]{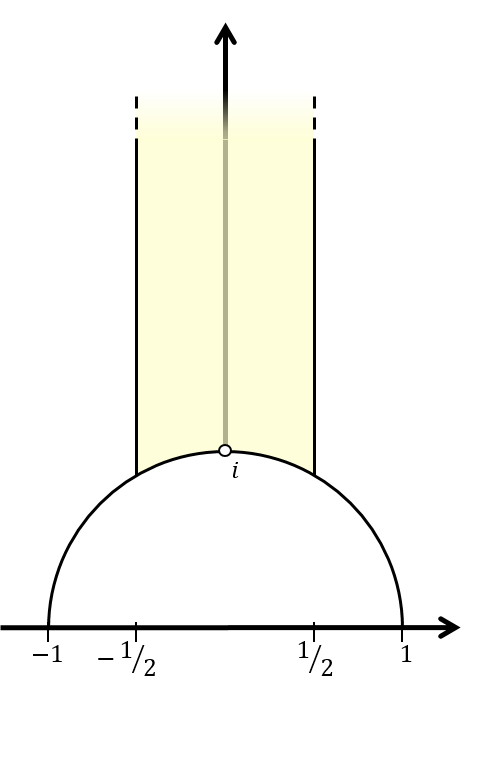}
\par\end{centering}
\caption{\label{fig: fund dom SL(2,Z)}$\protect\symfund 2$, a fundamental
domain for $\protect\sl 2\left(\protect\ZZ\right)$ in $P_{2}$ (the
hyperbolic upper half plane).}
\end{figure}

\begin{notation}
\label{nota: Fundamental domains for SL(Z)} We denote this fundamental
domain in $P_{i}$ by $\symfund i$. Also, we denote by $\groupfund i$
the fundamental domain in $\sl i\left(\RR\right)$ for $\sl i\left(\ZZ\right)$
that is obtained as 
\[
\groupfund i=\bigcup_{z\in\symfund i}K_{z}\cdot z,
\]
(\cite[Theorem 7.10 and Proposition 7.13]{HK_WellRoundedness}) where
for every $z\in\symfund i$, $K_{z}\subset\so i\left(\RR\right)$
is a fundamental domain for the finite group of rotations in $\so i\left(\RR\right)$
that preserve $\lat_{z}$. 
\end{notation}

Indeed, we will use $\symfund i$ as a spread model for $\shapespace i$,
and $\groupfund i$ as a spread model for $\latspace i$. Here is
the full list of spread models representing the spaces that appear
in Theorem \ref{thm: A counting}: 
\begin{prop}[{\cite[Prop. 8.1]{HK_WellRoundedness}}]
\label{prop: spread models that we need}There exists a spread model
$K^{\prime}\subset K$ for the space $\gras{d,n}\cong K^{\dprime}\backslash K=\wc/G^{\dprime}$.
Moreover, the following subsets of $\sl n\left(\RR\right)$ are spread
models for the corresponding spaces: 
\begin{itemize}
\item $\groupfund i\subset G_{i}$ for the spaces $\unilatspace i=G_{i}/G_{i}\left(\ZZ\right)$
for $i=d,n-d$ 
\item $\symfund i\subset P_{i}$ for the spaces $\shapespace i\diffeo K_{i}\backslash G_{i}/G_{i}\left(\ZZ\right)\diffeo P_{i}/G_{i}\left(\ZZ\right)$ 
\item $K^{\prime}\brac{\groupfund d\times\idmat{n-d}}\subset K^{\prime}G_{d}^{\dprime}=\wc_{d}\subset\wc$
for the space $\latspace{d,n}\diffeo\wc/G_{d}^{\dprime}\left(\ZZ\right)G_{n-d}^{\dprime}\left(\RR\right)$
\item $K^{\prime}\brac{\groupfund d\times\groupfund{n-d}}\subset K^{\prime}G^{\dprime}=\wc$
for the space $\pairspace{d,n}\diffeo\wc/G^{\dprime}\left(\ZZ\right)$
\end{itemize}
The measures we have defined on this spaces in the Introduction are
the unique measures from part 2 of Proposition \ref{prop: spread models - properties}. 
\end{prop}

\section{\label{sec: Integral matrices representing primitive vectors}Correspondence
between integral matrices and primitive lattices}

The goal of this section is to translate Theorem \ref{thm: A counting}
into a counting problem of integral matrices. The first step is to
establish a correspondence between primitive $d$-lattices and integral
matrices in a fundamental domain of the following discrete group of
$\sl n\brac{\RR}$: 
\[
\Lat:=\left(N^{\prime}\rtimes G^{\dprime}\right)\left(\ZZ\right)=\left[\begin{array}{cc}
\sl d\left(\ZZ\right) & \ZZ^{d,n-d}\\
0 & \sl{n-d}\left(\ZZ\right)
\end{array}\right].
\]

\begin{prop}
\label{prop: primitive sublattices correspond to integral matrices}There
exists a bijection $\lat\leftrightarrow\ga_{\lat}$ between oriented
primitive $d$-lattices and integral matrices in a fundamental domain
of $\sl n\left(\RR\right)\curvearrowleft\Lat$, that sends a $d$-lattice
$\lat$ to $\ga_{\lat}$, the unique integral matrix in the fundamental
domain whose first $d$ columns span $\lat$. 
\end{prop}

\begin{proof}
The direction $\Leftarrow$ is simple: given $\ga\in\funddom\cap\sl n\left(\ZZ\right)$,
its columns span $\ZZ^{n}$ hence by definition its first $d$ columns
span a primitive $d$\textendash lattice. In the opposite direction,
let $\base$ be basis for $\lat$. Since $\lat$ is primitive, $\base$
can be completed to a basis of $\ZZ^{n}$; let $\ga\in\sl n\left(\ZZ\right)$
be a matrix having this basis in its columns, with $\base$ in the
first $d$ columns. The orbit $\ga\cdot\Lat$, whose elements consist
of integral matrices having a basis for $\lat$ in their first $d$
columns, meets $\funddom$ in a single point, $\ga_{\lat}$. 
\end{proof}
Let us construct an explicit fundamental domain for $\Lat$. Let 
\[
\cube:=\text{the unit cube \ensuremath{\left(-1/2,1/2\right]^{d\left(n-d\right)}}}
\]
and, recalling Notation \ref{nota: Fundamental domains for SL(Z)},
set 
\begin{equation}
\funddom:=K^{\prime}\parby{G^{\dprime}}{\groupfund d\times\groupfund{n-d}}A^{\prime}\parby{N^{\prime}}{\cube}.\label{eq: fund dom for disc grp}
\end{equation}
It is easy to see  that $\funddom$ is a fundamental domain for the
right action of $\disgrp$ on $\sl n\left(\RR\right)$. The next goal
is to reduce the proof of Theorem \ref{thm: A counting} into a problem
of counting integral matrices in subsets of $\sl n\left(\RR\right)$,
and specifically of $\funddom$. We begin by defining these subsets. 
\begin{notation}
\label{nota: sets we count in}For $T>0$, $\sphereset\subseteq\gras{d,n}$,
$\groupset\times\groupfacset\subseteq\unilatspace d\times\unilatspace{n-d}$,
$\symset\times\factorset\subseteq\shapespace d\times\shapespace{n-d}$,
$\latset\subseteq\latspace{d,n}$ and $\pairset\subseteq\pairspace{d,n}$,
consider 
\[
\funddom_{T}\brac{\sphereset,\symset,\factorset}=\,\funddom\,\cap\,\left\{ g=k^{\prime}k^{\dprime}p^{\dprime}a^{\prime}n^{\prime}:\begin{matrix}k^{\prime}\in\parby{K^{\prime}}{\sphereset},a^{\prime}\in\parby{A^{\prime}}{\left[0,T\right]},\\
p^{\dprime}\in\parby{P^{\dprime}}{\symset\times\factor{\symset}}
\end{matrix}\right\} =\parby{K^{\prime}}{\sphereset}K^{\dprime}\parby{P^{\dprime}}{\symset\times\factorset}\parby{A^{\prime}}{\left[0,T\right]}\parby{N^{\prime}}{\cube},
\]
\[
\funddom_{T}\brac{\sphereset,\groupset,\groupfacset}=\,\funddom\cap\left\{ g=k^{\prime}g^{\dprime}a^{\prime}n^{\prime}:\begin{matrix}k^{\prime}\in\parby{K^{\prime}}{\sphereset},a^{\prime}\in\parby{A^{\prime}}{\left[0,T\right]},\\
g^{\dprime}\in\parby{G^{\dprime}}{\groupset\times\factor{\groupset}}
\end{matrix}\right\} =\parby{K^{\prime}}{\sphereset}\parby{G^{\dprime}}{\groupset\times\groupfacset}\parby{A^{\prime}}{\left[0,T\right]}\parby{N^{\prime}}{\cube},
\]
\[
\funddom_{T}\brac{\latset,\groupfacset}=\,\funddom\,\cap\,\left\{ g=qa^{\prime}n^{\prime}:q\in\parby{\wc}{\latset\times\factor{\groupset}},a^{\prime}\in\parby{A^{\prime}}{\left[0,T\right]}\right\} =\parby{\brac{\wc_{d}}}{\latset}\parby{\brac{G_{n-d}^{\dprime}}}{\groupfacset}\parby{A^{\prime}}{\left[0,T\right]}\parby{N^{\prime}}{\cube},
\]
and
\[
\funddom_{T}\brac{\pairset}=\,\funddom\,\cap\,\left\{ g=qa^{\prime}n^{\prime}:q\in\parby{\wc}{\pairset},a^{\prime}\in\parby{A^{\prime}}{\left[0,T\right]}\right\} =\parby{\wc}{\pairset}\parby{A^{\prime}}{\left[0,T\right]}\parby{N^{\prime}}{\cube}.
\]
\end{notation}

Now the following is immediate from Proposition \ref{prop: primitive sublattices correspond to integral matrices}
and Proposition \ref{prop: explicit RI coordinates of g}:
\begin{cor}
\label{cor: the sets we should count in}Consider the correspondence
$\lat\leftrightarrow\ga_{\lat}$ between primitive $d$\textendash lattices
and matrices in $\funddom\cap\sl n\left(\ZZ\right)$, and let $T>0$.

\begin{enumerate}
\item For $\sphereset\subseteq\gras{d,n}$ and $\symset\times\factor{\symset}\subseteq\shapespace d\times\shapespace{n-d}$,
\[
\funddom_{T}\brac{\sphereset,\symset,\factorset}\cap\sl n\left(\ZZ\right)=\cbrac{\ga_{\lat}:\covol{\lat}\leq e^{T},V_{\lat}\in\sphereset,\shape{\lat}\in\symset,\shape{\factor{\lat}}\in\factorset};
\]
\item For $\sphereset\subseteq\gras{d,n}$ and $\groupset\times\groupfacset\subseteq\unilatspace d\times\unilatspace{n-d}$,
\begin{align*}
\funddom_{T}\brac{\sphereset,\groupset,\groupfacset}\cap\sl n\left(\ZZ\right) & =\cbrac{\ga_{\lat}:\covol{\lat}\leq e^{T}\text{, }V_{\lat}\in\sphereset,\unisimlat{\lat}\in\groupset,\unisimlat{\factor{\lat}}\in\groupfacset};
\end{align*}
\item For $\latset\subseteq\latspace{d,n}$ and $\groupfacset\subseteq\unilatspace{n-d}$,
\[
\funddom_{T}\brac{\latset,\groupfacset}\cap\sl n\left(\ZZ\right)=\cbrac{\ga_{\lat}:\covol{\lat}\leq e^{T}\text{, }\unilat{\lat}\in\latset,\unisimlat{\factor{\lat}}\in\groupfacset};
\]
\item For $\pairset\subseteq\pairspace{d,n}$, 
\begin{align*}
\funddom_{T}\brac{\pairset}\cap\sl n\left(\ZZ\right) & =\cbrac{\ga_{\lat}:\covol{\lat}\leq e^{T},\unilat{\brac{\lat,\factor{\lat}}}\in\pairset}.
\end{align*}
\end{enumerate}
\end{cor}

\section{\label{sec: Counting lattices up the cusp}Neglecting lattices up
the cusp: reduction to counting in compact sets}

We are now at the point where we have reduced the proof of Theorem
\ref{thm: A counting} to a problem of counting integral matrices
inside the subsets of $\sl n\left(\RR\right)$ that are defined in
Notation \ref{nota: sets we count in}. The main obstacle in handling
this counting problem, is that these sets are not compact: while their
$A^{\prime}$ component is bounded in $\left[0,T\right]$, their $A^{\dprime}$
component is not bounded from above. This section is devoted to tackling
this issue, by reducing to counting in compact subsets of $\funddom$.
These subsets will be obtained by truncating the $A^{\dprime}$ coordinates
of the sets from Notation \ref{nota: sets we count in}; truncating
the coordinates in $A_{d}^{\dprime}$ will result in bounding the
$d$ component, and truncating the coordinates in $A_{n-d}^{\dprime}$
will result in bounding the $n-d$ component:
\begin{notation}
\label{nota: truncated}For every $\Svec=\brac{S_{1},\ldots,S_{d}}>\underline{0}$,
$\Wvec=\brac{W_{1},\dots,W_{n-d}}>\underline{0}$ and a subset $\Gset\subset G$,
let $\trunc{\Gset}{\Svec,\Wvec}$ denote the subset 
\[
\Gset\cap\cbrac{g:\pi_{A^{\prime\prime}}\left(g\right)=a_{\svec,\wvec},s_{i}\leq S_{i}\,\forall1\leq i\leq d,w_{j}\leq W_{j}\,\forall1\leq j\leq n-d},
\]
where $\pi_{A^{\prime\prime}}$ is the projection to the $A^{\prime\prime}$
component. If $\pi_{A_{n-d}^{\dprime}}\brac{\Gset}=e$, then we denote
$\trunc{\Gset}{\Svec}$, and similarly if $\pi_{A_{d}^{\dprime}}\brac{\Gset}=e$
then we denote $\trunc{\Gset}{\Wvec}$.  
\end{notation}

Recall the fundamental domain $\funddom\subset G$ for $\Lat$ appearing
in Formula \ref{eq: fund dom for disc grp}. Accordingly with Notation
\ref{nota: truncated}, we let 
\[
\trunc{\funddom_{T}}{\Svec,\Wvec}:=K^{\prime}\left[\begin{smallmatrix}\trunc{\groupfund d}{\Svec}\\
 & \trunc{\groupfund{n-d}}{\Wvec}
\end{smallmatrix}\right]A_{T}^{\prime}\parby{N^{\prime}}{\cube}.
\]

\begin{rem}
\label{rem: measure of cusp}It is not hard to see that $\mu\brac{\funddom_{T}-\trunc{\funddom_{T}}{\Svec,\Wvec}}$
is in $O\brac{e^{nT-\min\left\{ S_{\min},W_{\min}\right\} }}$, where
$S_{\min}=\min_{j}S_{j}$ and $W_{\min}=\min_{j}W_{j}$. 
\end{rem}

The goal of this section is to prove the following. 
\begin{prop}
\label{cor: very few SL(n,Z) points up the cusp}Let $\funddom$ be
our fundamental domain for $\Lat$ in $G$, and $\underline{\s}=\left(\s_{1},\ldots,\s_{d-1}\right)$,
$\underline{\om}=\left(\om_{1},\ldots,\om_{n-d-1}\right)$ where $0<\s_{i},\om_{i}<1$
$\forall i$. Then for every $\e>0$
\[
\#\left|\left(\funddom_{T}-\trunc{\funddom_{T}}{\underline{\sigma}T,\underline{\infty}}\right)\cap\sl n\left(\ZZ\right)\right|=O_{\e}\left(e^{T\left(n-\sigma_{\min}+\e\right)}\right),\:\#\left|\left(\funddom_{T}-\trunc{\funddom_{T}}{\underline{\infty},\underline{\omega}T}\right)\cap\sl n\left(\ZZ\right)\right|=O_{\e}\left(e^{T\left(n-\om_{\min}+\e\right)}\right)
\]
and in particular
\[
\#\left|\left(\funddom_{T}-\trunc{\funddom_{T}}{\underline{\sigma}T,\underline{\omega}T}\right)\cap\sl n\left(\ZZ\right)\right|=O_{\e}\left(e^{T\left(n-\min\left\{ \sigma_{\min},\om_{\min}\right\} +\e\right)}\right),
\]
where $\sigma_{\min}=\min\sigma_{i}$, $\omega_{\min}=\min\omega_{i}$.
\end{prop}

\begin{proof}
Let $\ga_{\lat}=ka_{\underline{s},\wvec}^{\prime\prime}a_{t}^{\prime}n\in\funddom_{T}\cap\sl n\left(\ZZ\right)$.
In what follows, $\lat^{i}$ is the lattice spanned by the first $i$
columns of $\ga_{\lat}$, $\brac{\factor{\lat}}^{j}<\factor{\lat}$
is the lattice spanned by the first $j$ columns of the $n\times(n-d)$
matrix obtained by projecting the columns of $\ga_{\lat}$ to $\perpen{V_{\lat}}$,
and $\left(\perpen{\lat}\right)^{k}<\perpen{\lat}$ is the lattice
spanned by the first $k$ columns of the matrix that represents $\perpen{\lat}$
inside $\groupfund{n-d}$.

Recall that $\ga_{\lat}\in\funddom_{T}-\funddom_{T}^{\left(\underline{\sigma}T,\underline{\infty}\right)}$
if and only if $\exists i\in\left\{ 1,\ldots,d-1\right\} $ for which
$s_{i}\geq\sigma_{i}T$. This implies (using Proposition \ref{prop: explicit RI coordinates of g}
and the fact that $\lat^{i}$ is integral) that $1\leq\covol{\lat^{i}}=e^{\frac{it}{d}-\frac{\sigma_{i}t}{2}}\leq e^{\brac{\frac{i}{d}-\frac{\sigma_{i}}{2}}T}$.
For the remaining $i\neq i^{\prime}\in\left\{ 1,\ldots,d-1\right\} $,
$1\leq\covol{\lat^{i^{\prime}}}=e^{\frac{it}{d}}\leq e^{\frac{i}{d}T}$. 

By \cite[Prop. 6.4]{HK_gcd}, the number of such possible lattices
$\lat$ is $e^{T\left(n-\sigma+\e\right)}$ where $\sigma=\sum\sigma_{i}$.
Thus, the number of $\sl n\left(\ZZ\right)$ elements $\ga$ in $\funddom_{T}-\funddom_{T}^{\left(\underline{\sigma}T,\underline{\infty}\right)}$
is bounded by 
\begin{align*}
\;\, & \#\left|\bigcup_{\substack{\underline{u}=\left(u_{1},\ldots,u_{d-1}\right)\\
\in\left\{ 0,1\right\} ^{d-1}-\left\{ 0\right\} 
}
}\cbrac{\lat:\forall i,\,\covol{\lat_{v}^{i}}\in\sbrac{1,e^{T\brac{\frac{i}{d}-\frac{\sigma_{i}u_{i}}{2}}}}}\right|\\
= & \sum_{\underline{u}\in\left\{ 0,1\right\} ^{d-1}-\left\{ 0\right\} }O_{\e}\brac{e^{T\brac{n-\sigma+\e}}}=O_{\e}\brac{e^{T\brac{n-\sigma_{\min}+\e}}}
\end{align*}
where $\sigma_{\min}=\min\left\{ \sigma_{i}\right\} $, as $\e$ is
arbitrary.

Now recall that $\ga_{\lat}\in\funddom_{T}-\funddom_{T}^{\left(\underline{\infty},\underline{\omega}T\right)}$
if and only if $\exists j\in\left\{ 1,\ldots,n-d-1\right\} $ for
which $w_{j}\geq\omega_{j}T$. By Lemma \ref{lem: covolume of sub-lattices of factor and perp}
and Proposition \ref{prop: dual+factor =00003D orthogonal}, $\covol{\brac{\factor{\lat}}^{j}}$
is proportional to the quotient $\covol{\brac{\perpen{\lat}}^{n-d-j}}/\covol{\lat}$,
with a positive constant that depends only on $d$ and $n$. Then,
by Proposition \ref{prop: explicit RI coordinates of g} and Lemma
\ref{lem: covolume of sub-lattices of factor and perp},
\[
\widetilde{C}e^{-t}\cdot\covol{\brac{\perpen{\lat}}^{n-d-j}}=\widetilde{C}\cdot\frac{\covol{\brac{\perpen{\lat}}^{n-d-j}}}{\covol{\lat}}\leq\covol{\brac{\factor{\lat}}^{j}}=e^{-\frac{jt}{n-d}-\frac{w_{j}}{2}},
\]
where $\widetilde{C}>0$. Hence, up to an additive constant that becomes
negligible when $t$ is large, $w_{j}\geq\omega_{j}T$ implies
\[
1\leq\covol{\brac{\perpen{\lat}}^{n-d-j}}\leq e^{t-\frac{jt}{n-d}-\frac{w_{j}}{2}}\leq e^{T\brac{1-\frac{j}{n-d}-\frac{\omega_{j}}{2}}}.
\]
For the remaining $j\neq j^{\prime}\in\left\{ 1,\ldots,n-d-1\right\} $,
$1\leq\covol{\lat^{n-d-j^{\prime}}}\leq e^{t\brac{1-\frac{j^{\prime}}{n-d}}}\leq e^{T\brac{1-\frac{j^{\prime}}{n-d}}}$.
By \cite[Cor. 6.4]{HK_gcd}, the number of such possible lattices
$\lat$ is $e^{T\left(n-\omega+\e\right)}$ where $\omega=\sum\om_{i}$.
Thus, by similar considerations, the number $\#\vbrac{\brac{\funddom_{T}-\funddom_{T}^{\left(\underline{\infty},\underline{\omega}T\right)}}\cap\sl n\left(\ZZ\right)}$
is in $O\brac{e^{T\left(n-\omega_{\min}+\e\right)}}$, where $\omega_{\min}=\min\omega_{i}$. 
\end{proof}

\section{\label{sec: Almost  a proof}Almost a proof of the main theorem}

In this section, we prove Theorem \ref{thm: A counting} based on
a proposition for counting lattice points in $\sl n\left(\RR\right)$
that we state here and prove in Section \ref{sec: Technical Proof}.
The statement includes a parameter $\errexp$ of lattices in Lie groups,
that we now define. 

In certain Lie groups, among which algebraic simple non-compact Lie
groups $G$, there exists $p\in\mathbb{N}$ for which the matrix coefficients
$\dbrac{\pi_{G/\Lat}^{0}u,v}$ are in $L^{p+\e}\left(G\right)$ for
every $\e>0$, with $u,v$ lying in a dense subspace of $L_{0}^{2}\left(G/\Lat\right)$
(see \cite[Thm 5.6]{GN_book}). Let $p\left(\Lat\right)$ be the smallest
among these $p$'s, and denote 
\[
m\left(\Lat\right)=\begin{cases}
1 & \text{if \ensuremath{p=2},}\\
2\left\lceil p\left(\Lat\right)/4\right\rceil  & \text{otherwise.}
\end{cases}
\]
Define: 
\[
\errexp\left(\Lat\right)=\frac{1}{2m\left(\Lat\right)\left(1+\dim G\right)}\in\left(0,1\right),
\]

\begin{prop}
\label{thm: Counting with S(T) and W(T)}Let $\Lat<\sl n\left(\RR\right)$
be a lattice with $\errexp=\errexp\left(\Lat\right)$. Set $\lm_{n}=\frac{n^{2}}{2\left(n^{2}-1\right)}$.
If $\pairset\subseteq\pairspace{d,n}$ is a BCS, then the following
holds for every $0<\e<\errexp$: 
\begin{enumerate}
\item For $\Svec=\left(S_{1},\dots,S_{d-1}\right)$ and $\Wvec=\left(W_{1},\dots,W_{n-d-1}\right)$,
$\sumS=\sum S_{i}$ and $\sumW=\sum W_{i}$, and every $T\geq\frac{\sumS+\sumW}{n\lm_{n}\errexp\left(\Lat\right)}+O_{\pairset}(1)$:
\[
\#\brac{\trunc{\funddom_{T}}{\Svec,\Wvec}\brac{\pairset}\cap\Lat}=\frac{\mu\brac{\trunc{\funddom_{T}}{\Svec,\Wvec}\brac{\pairset}}}{\mu\left(\Lat\backslash G\right)}+O_{\pairset,\e}\brac{e^{\left(\sumS+\sumW\right)/\lm_{n}}e^{nT\left(1-\errexp+\e\right)}}.
\]
\item For every $\delta\in\left(0,\errexp-\e\right)$ , $T\geq O_{\pairset}(1)$,
$\Svec\left(T\right)=\left(S_{1}\left(T\right),\dots,S_{d-1}\left(T\right)\right)$
and $\Wvec\left(T\right)=\left(W_{1}\left(T\right),\dots,W_{n-d-1}\left(T\right)\right)$
such that $\sumS\left(T\right)+\sumW\left(T\right)=\sum S_{i}\left(T\right)+\sum W_{i}\left(T\right)\leq n\dl\lm_{n}T+O_{\pairset}(1)$:
\[
\#\brac{\trunc{\funddom_{T}}{\Svec\left(T\right),\Wvec\left(T\right)}\brac{\pairset}\cap\Lat}=\frac{\mu\brac{\trunc{\funddom_{T}}{\Svec\left(T\right),\Wvec\left(T\right)}\brac{\pairset}}}{\mu\left(\Lat\backslash G\right)}+O_{\pairset,\e}\brac{e^{nT\left(1-\errexp+\dl+\e\right)}}.
\]
When only $\Svec$ is bounded the condition on $\sumS\left(T\right)+\sumW\left(T\right)$
becomes $\sumW\left(T\right)\leq n\dl\lm_{n}T+O_{\pairset}(1)$,
and when only $\Wvec$ is bounded the condition becomes $\sumS\left(T\right)\leq n\dl\lm_{n}T+O_{\pairset}(1)$.
\end{enumerate}
\end{prop}

\begin{proof}[Proof of Theorem \ref{thm: A counting} assuming Proposition \ref{thm: Counting with S(T) and W(T)}]
By Corollary \ref{cor: the sets we should count in}, when setting
$T=\log X$, the quantities we seek to estimate in parts 1,2,3 and
4 of the theorem is in one to one correspondence with the integral
matrices in the following subsets of $\sl n\left(\RR\right)$, respectively:
1) $\funddom_{T}\brac{\sphereset,\symset,\factorset}$, 2) $\funddom_{T}\brac{\sphereset,\groupset,\groupfacset}$,
3) $\funddom_{T}\brac{\latset,\groupfacset}$, or 4) $\funddom_{T}\brac{\pairset}$.
Observe that, according to Section \ref{subsec: RI Haar measure}
on the measures on the RI components, to Proposition \ref{prop: spread models that we need}
on the spread models for $\shapespace i$, $\unilatspace i$, $\latspace{d,n}$
and $\pairspace{d,n}$, to \ref{prop: spread models - properties}(2)
for the agreement of the measures on the lattice spaces and the measures
on their spread models, and to our choice of measure on $\so i\left(\RR\right)$,
indeed the main terms in Theorem \ref{thm: A counting} are the volumes
of these sets, divided by the measure of $\sl n\left(\RR\right)/\sl n\left(\ZZ\right)$
which is $\vol\brac{\unilatspace n}=\prod_{i=2}^{n}\zeta\left(i\right)$. 

First we claim that it is sufficient to prove part \ref{enu: MainThm_restricted pairs}
of the theorem, and then the other parts will follow. Indeed, the
family $\funddom_{T}\brac{\sphereset,\symset,\factorset}$ appearing
in part \ref{enu: MainThm_restricted shape} is a special case of
the family $\funddom_{T}\brac{\sphereset,\groupset,\groupfacset}$
appearing in part \ref{enu: MainThm_restricted G''}, by taking $\groupset,\groupfacset$
to be the lifts of $\symset,\factorset$. According to Lemma \ref{lem: lift to G''}(1),
$\groupset,\groupfacset$ are BCS's when $\symset,\factorset$ are,
and $\text{\ensuremath{\vol\brac{\groupset}}}=\vol\brac{\symset}\halfK d$
(similarly for $\groupfacset$ and $n-d$). In particular, the main
term provided in part \ref{enu: MainThm_restricted shape} of the
theorem for $\sphereset,\symset,\factorset$ coincides with the one
provided in part \ref{enu: MainThm_restricted G''} for $\sphereset$
and the lifts of $\symset,\factorset$. Hence, part \ref{enu: MainThm_restricted shape}
is a consequence of part \ref{enu: MainThm_restricted G''}. Similarly,
due to Lemma \ref{lem: lift to G''}(3) and (4), parts \ref{enu: MainThm_restricted G''}
and \ref{enu: MainThm_restricted components of pairs} respectively
are a consequence of part \ref{enu: MainThm_restricted pairs}. We
conclude that it is sufficient to prove part \ref{enu: MainThm_restricted pairs}
of the theorem. 

Let $0<\e<\errexp_{n}$ and $0<\dl<\errexp_{n}-\e$. Suppose first
that the $d$ component of $\pairset\subseteq\pairspace{d,n}$ is
not bounded, and $n-d$ component is bounded. Recall $\lm_{n}=\frac{n^{2}}{2\left(n^{2}-1\right)}$
and let $\underline{\sigma}_{d}=\brac{\frac{\dl n\lm_{n}}{d-1}-\e}\cdot\Onevec_{d-1}$.
Note that the sum of the coordinates of $\underline{\sigma}_{d}$
is $\dl n\lm_{n}-(d-1)\e$, which, for $T$ large enough, is smaller
than $\dl n\lm_{n}+O_{\pairset}(1)/T$ (aiming to satisfy the condition
in  Proposition \ref{thm: Counting with S(T) and W(T)}). Since the
$n-d$ component of $\pairset$ is bounded, there exists $\Wvec=\Wvec_{\pairset}$
such that $\funddom_{T}\brac{\pairset}=\trunc{\funddom_{T}}{\infty,\Wvec_{\pairset}}\brac{\pairset}$.
Using this, along with Propositions \ref{cor: very few SL(n,Z) points up the cusp}
and \ref{thm: Counting with S(T) and W(T)}, we have
\[
\#\brac{\sl n\left(\ZZ\right)\cap\funddom_{T}\brac{\pairset}}=\#\brac{\sl n\left(\ZZ\right)\cap\trunc{\funddom_{T}}{\infty,\Wvec_{\pairset}}\brac{\pairset}}
\]
\[
=\#\brac{\sl n\left(\ZZ\right)\cap\trunc{\funddom_{T}}{\underline{\sigma}_{d}T,\Wvec_{\pairset}}\brac{\pairset}}+O_{\e}\brac{e^{nT\left(1-\frac{\dl\lm_{n}}{d-1}+\e\right)}}
\]
\[
=\mu\brac{\trunc{\funddom_{T}}{\underline{\sigma}_{d}T,\Wvec_{\pairset}}\brac{\pairset}}+O_{\pairset,\e}\brac{e^{nT\left(1-\errexp_{n}+\dl+\e\right)}}+O_{\e}\brac{e^{nT\left(1-\frac{\dl\lm_{n}}{d-1}+\e\right)}}
\]
According to Remark \ref{rem: measure of cusp}, we may replace the
measure of $\trunc{\brac{\funddom}}{\underline{\sigma}_{d}T,\Wvec_{\pairset}}_{T}\brac{\pairset}$
in the expression above with the measure of $\funddom{}_{T}\brac{\pairset}$,
since the difference in measures is $O\brac{e^{nT\left(1-\frac{\dl\lm_{n}}{d-1}\right)}}$,
which is swallowed in the error term. Then we choose $\dl$ that will
balance the two error terms above, i.e.\ $\dl$ that satisfies: $1-\errexp_{n}+\dl=1-\frac{\dl\lm_{n}}{d-1}$;
this means $\dl=\errexp_{n}\cdot\brac{1-\frac{\lm_{n}}{d-1+\lm_{n}}}=\errexp_{n}\cdot\brac{1-\frac{n^{2}}{2\left(d-1\right)\left(n^{2}-1\right)+n^{2}}}$.
We conclude that in the case where only the $d$ component of $\pairset$
is bounded, then $\#\brac{\sl n\left(\ZZ\right)\cap\funddom_{T}\brac{\pairset}}=\mu\brac{\funddom_{T}\brac{\pairset}}+O_{\pairset,\e}\brac{e^{nT\brac{1-\frac{\errexp_{n}n^{2}}{2\left(d-1\right)\left(n^{2}-1\right)+n^{2}}+\e}}}$.
Similarly, when only the $n-d$ component is bounded, we define $\underline{\om}_{n-d}:=\frac{\dl n\lm_{n}}{n-d-1}\cdot\Onevec_{n-d-1}$
and the error term changes to $O_{\pairset,\e}\brac{e^{nT\brac{1-\frac{\errexp_{n}n^{2}}{2\left(n-d-1\right)\left(n^{2}-1\right)+n^{2}}+\e}}}$;
when both $d$ and $n-d$ components are non bounded, the error term
changes to $O_{\pairset,\e}\brac{e^{nT\brac{1-\frac{\errexp_{n}n^{2}}{2\left(\max\left\{ d,n-d\right\} -1\right)\left(n^{2}-1\right)+n^{2}}+\e}}}$.
When both $d$ and $n-d$ components are bounded, we use only (the
first part of) Proposition \ref{thm: Counting with S(T) and W(T)},
and obtain an error term of $O_{\pairset,\e}\brac{e^{nT\left(1-\errexp+\e\right)}}$.
\end{proof}

\section{\label{sec: Technical Proof}Proof of technical counting proposition}

the goal of this section is to prove Proposition \ref{thm: Counting with S(T) and W(T)}.
We need the following from \cite{GN1}:
\begin{defn}
\label{def: well roundedness}Let $G$ be a Lie group with a Borel
measure $\mu$, and let $\cbrac{\nbhd{\e}{}}_{\e>0}$ be a family
of identity neighborhoods in $G$. Assume $\left\{ \Gset_{T}\right\} _{T>0}\subset G$
is a family of measurable domains and denote 
\[
\Gset_{T}^{+}\left(\e\right):=\nbhd{\e}{}\Gset_{T}\nbhd{\e}{}=\bigcup_{u,v\in\nbhd{\e}{}}u\,\Gset_{T}\,v,
\]
\[
\Gset_{T}^{-}\left(\e\right):=\bigcap_{u,v\in\nbhd{\e}{}}u\,\Gset_{T}\,v.
\]
The family $\left\{ \Gset_{T}\right\} $ is \emph{Lipschitz well-rounded
(LWR)} with (positive) parameters $\brac{\mathcal{C},T_{0}}$ if for
every $0<\e<1/\mathcal{C}$ and $T>T_{0}$: 
\begin{equation}
\mu\left(\Gset_{T}^{+}\left(\e\right)\right)\leq\left(1+\mathcal{C}\e\right)\:\mu\left(\Gset_{T}^{-}\left(\e\right)\right).\label{eq:LWReq}
\end{equation}
The parameter $\mathcal{C}$ is called the \emph{Lipschitz constant}
of the family $\left\{ \Gset_{T}\right\} $. 
\end{defn}

The definition above allows any family $\left\{ \nbhd{\e}{}\right\} _{\e>0}$
of identity neighborhoods; in this paper we shall restrict to the
following:
\begin{assumption}
\label{assu: Our O_e }We will assume that $\nbhd{\e}G=\exp\left(\ball{\e}\right)$,
where $B_{\e}$ is an origin-centered $\e$-ball inside the Lie algebra
of $G$, and $\exp$ is the Lie exponent. 
\end{assumption}

\begin{rem}
\label{rem: LWR set}We allow the case of a constant family $\left\{ \Gset_{T}\right\} =\Gset$:
we say that $\Gset$ is a Lipschitz well rounded \emph{set} (as apposed
to a Lipschitz well rounded \emph{family}) with parameters $\brac{\mathcal{C},\e_{0}}$
if  $\mu\brac{\Gset^{+}\left(\e\right)}\leq\brac{1+\mathcal{C}\e}\:\mu\brac{\Gset^{-}\left(\e\right)}$
for every $0<\e<\e_{0}$. It is proved in \cite[Prop. 3.5]{HK_WellRoundedness}
that if a set $\Gset$ is BCS and bounded, then it is LWR. 
\end{rem}

\begin{thm}[{\cite[Theorems 1.9, 4.5, and Remark 1.10]{GN1}}]
\label{thm: GN Counting thm}Let $G$ be an algebraic simple Lie
group with Haar measure $\mu$, and let $\Lat<G$ be a lattice. Assume
that $\{\Gset_{T}\}\subset G$ is a family of finite-measure domains
which satisfy $\mu(\Gset_{T})\to\infty$ as $T\to\infty$. If the
family $\{\Gset_{T}\}$ is Lipschitz well-rounded with parameters
$(C_{\Gset},T_{0})$, then $\exists T_{1}>0$ such that for every
$\delta>0$ and $T>T_{1}$:
\[
\#\brac{\Gset_{T}\cap\Lat}-\mu\brac{\Gset_{T}}/\mu\left(G/\Lat\right)=O_{G,\Lat,\delta}\brac{C_{\Gset}^{\frac{\dim G}{1+\dim G}}\cdot\mu\brac{\Gset_{T}}^{1-\errexp\left(\Lat\right)+\delta}},
\]
where $\mu\left(G/\Lat\right)$ is the measure of a fundamental domain
of $\Lat$ in $G$ and $\errexp\left(\Lat\right)\geq\frac{1}{2\left(1+\dim G\right)}$
is a parameter that depends on the rate of decay of the matrix coefficients
of the $G$-representation on $L_{0}^{2}\left(G/\Lat\right)$. The
parameter $T_{1}$ is such that $T_{1}\geq T_{0}$ and for every $T\geq T_{1}$
\begin{equation}
C_{\Gset}^{\frac{\dim G}{1+\dim G}}=O_{G,\Lat}\brac{\mu\left(\Gset_{T}\right)^{\errexp\left(\Lat\right)}}.\label{eq: def of T_1 in GN thm}
\end{equation}
\end{thm}

\begin{proof}[Proof of Proposition \ref{thm: Counting with S(T) and W(T)}]
To prove part 1 using Theorem \ref{thm: GN Counting thm}, we must
show that the family $\cbrac{\trunc{\funddom}{\left(\Svec,\Wvec\right)}_{T}\brac{\pairset}}_{T>0}$
is Lipschitz well rounded. Define the map
\[
\begin{array}{cccc}
\roundo: & \sl n\left(\RR\right) & \to & K\times A^{\prime}\times A^{\dprime}\times N^{\dprime}\times N^{\prime}\\
 & g=ka^{\prime}a^{\dprime}n^{\dprime}n^{\prime} & \mapsto & \left(k,a^{\prime},a^{\dprime},n^{\dprime},n^{\prime}\right)
\end{array}.
\]
Showing that the image $\cbrac{\roundo\brac{\trunc{\funddom}{\left(\Svec,\Wvec\right)}_{T}\brac{\pairset}}}_{T>0}$
under $\roundo$ is Lipschitz well rounded with Lipschitz constant
and $T_{0}$ that do not depend on $\Svec,\Wvec$, would imply, according
to Corollaries 4.3 in \cite{HK_WellRoundedness} and 10.5 in \cite{HK_gcd},
that $\trunc{\funddom}{\left(\Svec,\Wvec\right)}_{T}\brac{\pairset}$
is LWR with Lipschits constant that is $\prec_{\pairset}e^{2\left(\sumS+\sumW\right)}$
and $T_{0}$ that does not depend on $\Svec,\Wvec$. Now, 
\[
\roundo\brac{\trunc{\funddom}{\Svec,\Wvec}_{T}\brac{\pairset}}=\roundo\brac{\parby{\wc}{\trunc{\pairset}{\Svec,\Wvec}}}\times A_{\left[0,T\right]}^{\prime}\times\parby{N^{\prime}}{\cube}
\]
with $\roundo\brac{\parby{\wc}{\trunc{\pairset}{\Svec,\Wvec}}}\subset K\times A^{\dprime}\times N^{\dprime}$.
Each of the three factors in the above direct product is LWR: $\roundo\brac{\parby{\wc}{\trunc{\pairset}{\Svec,\Wvec}}}$
is LWR independently of $\Svec,\Wvec$ by \cite[Lemma 11.1]{HK_gcd}\footnote{Actually it was proved for $d=n-1$, but the proof is valid for any
$1<d<n$.}, $A_{\left[0,T\right]}^{\prime}$ is LWR by \cite[Proposition 9.6]{HK_gcd}
and $\parby{N^{\prime}}{\cube}$ is LWR by \cite[Proposition 3.5]{HK_WellRoundedness},
combined with the fact that $\cube$ is a bounded BCS. Since a direct
product of LWR families is itself LWR, we conclude that $\roundo\brac{\trunc{\funddom}{\left(\Svec,\Wvec\right)}_{T}\brac{\pairset}}$
is LWR, and therefore $\trunc{\funddom}{\left(\Svec,\Wvec\right)}_{T}\brac{\pairset}$
is. The error term in part 1 of the proposition is obtained from Theorem
\ref{thm: GN Counting thm}, while recalling that the volume of $\trunc{\funddom}{\left(\Svec,\Wvec\right)}_{T}\brac{\pairset}$
in $\sl n\brac{\RR}$ is of order $e^{nT}$, and that the Lipschitz
constant of this family is $\prec_{\pairset}e^{2\left(\sumS+\sumW\right)}$.
It remains to explain the lower bound on $T$, but this is obtained
from substituting the LWR parameters of $\trunc{\funddom}{\left(\Svec,\Wvec\right)}_{T}\brac{\pairset}$
into Equation \ref{eq: def of T_1 in GN thm}. 

For the proof of the second part, let us compute a bound on $\sumS,\sumW$
for which the error term established in part 1 of the proposition
remains smaller than the main term. Namely, that there exists $\ga\in\left(0,1\right)$
for which
\[
\brac{\sumS+\sumW}/\lm_{n}+\left(1-\errexp\left(\Lat\right)+\e\right)\cdot nT=\Lat\cdot nT.
\]
If $\delta$ denotes the number $\ga+\errexp\left(\Lat\right)-\e-1$,
we have that $\ga=\delta+1+\e-\errexp\left(\Lat\right)$. Then $\ga<1$
if and only if $\dl<\errexp\left(\Lat\right)-\e$, where $\errexp\left(\Lat\right)-\e$
is positive since $\errexp\left(\Lat\right)>\e$. We conclude that
for $0<\dl<\errexp\left(\Lat\right)-\e$ and $\sumS+\sumW=\sumS\left(T\right)+\sumW\left(T\right)<\dl\lm_{n}nT$,
the counting applies with an error term of order $e^{\ga nT}=e^{nT\left(1-\errexp+\dl+\e\right)}$.
As for the lower bound $T_{1}$ on $T$, in part 1 we got $\sumS+\sumW\leq n\lm_{n}\errexp\left(\Lat\right)T+O_{\pairset}\left(1\right)$;
so, combining both bounds on $\sumS+\sumW$ we get 
\[
\sumS+\sumW\leq\min\left\{ n\lm_{n}\dl T,\:n\lm_{n}\errexp\left(\Lat\right)T\right\} +O_{\latset}\left(1\right)=n\lm_{n}\dl T+O_{\pairset}\left(1\right)
\]
for $T$ large enough and $\delta\in\left(0,\errexp\left(\Lat\right)-\e\right)$.
This completes the proof. 
\end{proof}
\newpage

\appendix

\part*{Appendix: Dual lattices, factor Lattices}

This appendix is completely independent from the paper. It aims to
be a source for facts that are not necessarily unknown or hard to
prove, but can be hard to find in the literature. For example, while
the concept of dual lattices (sometimes referred to as \emph{polar}
lattices) is well known \cite{Cassels71}, and we have also found
the notes \cite{Regev_Lattices} very useful), the concept of factor
lattices, a term coined by Schmidt in \cite{Schmidt_68}, is much
less known, and in fact we have not found any reference for it. Another
aim of this appendix is to add the differential perspective of smooth
maps between spaces of lattices. For example, we will see that the
map sending a $d$-lattice to its dual is an auto-diffeomorphism of
the space of $d$-lattices inside $\RR^{n}$. 

In what follows, if the columns of a matrix $\base=\base_{n\times d}$
span a $d$-lattice $\lat<\RR^{n}$ (resp.\ a subspace $V<\RR^{n}$),
we say that $\base$ is a basis for the lattice $\lat$ (resp. \
the subspace $V$). Recall that the covolume of a $d$-lattice $\lat$,
which is the volume of the fundamental parallelepiped for $\lat$
in the linear subspace that it spans, equals $\brac{\vbrac{\det\brac{\base^{\transpose}\base}}}^{1/2}$
where $\base$ is (any) a basis for $\lat$. We will use an underscore
to denote that an object is being spanned by a set, so that $\lat_{\base}$
is the lattice spanned (over $\ZZ$) by $\base$, $V_{\lat}$ is the
linear space spanned (over $\RR$) by $\lat$, etc. 

\section{\label{sec: Dual Lattices}Dual lattices}

Given a $d$-lattice $\lat<\RR^{n}$, we define the \emph{dual lattice}
of $\lat$ as 
\[
\dual{\lat}=\left\{ y\in\sp{\RR}{\lat}:\forall x\in\lat,\dbrac{x,y}\in\ZZ\right\} .
\]
Note that it is contained in $V_{\lat}$, and that a priori, it is
unclear that $\dual{\lat}$ is indeed a lattice. 
\begin{prop}
\label{prop: dual lattice basis}The matrix $\based_{n\times d}:=\base\brac{\base^{\transpose}\base}^{-1}$
is basis for $\dual{\lat}$. 
\end{prop}

This proposition motivates the notation $\based=\dual{\base}$ for
$\based$ as above, as well as the name \emph{the dual basis} of $\base$.
The proof will make use of the following lemma:
\begin{lem}
\label{lem: dual basis properties}$\base$ and $\based$ span the
same $d$-dimensional linear space in $\RR^{n}$, and their columns
satisfy the orthonormality relation
\[
\left\langle b_{i},d_{j}\right\rangle =\begin{cases}
1 & i=j\\
0 & i\neq j
\end{cases}
\]
where $\brac{b_{i}}_{i=1}^{d}$ are the columns of $\base$ and $\brac{d_{i}}_{i=1}^{d}$
are the columns of $\based$.
\end{lem}

\begin{proof}
The matrix $\based$ is of the form $\base g$ where $g\in\gl d\left(\RR\right)$,
so $\base$ and $\based$ span the same linear space in $\RR^{n}$.
The orthonormality relation follows from the fact that $\base^{\transpose}\based=\idmat d$. 
\end{proof}
\begin{proof}[Proof of Proposition \ref{prop: dual lattice basis}]
In order to show $\lat_{\based}=\dual{\lat_{\base}}$, we will prove
two inclusions. For $\lat_{\based}\supseteq\dual{\lat_{\base}}$,
we first note that by Lemma \ref{lem: dual basis properties}, $V_{\based}\supseteq\dual{\lat_{\base}}$.
As a result, an element $y\in\dual{\lat_{\base}}$ is of the form
$y=\sum_{i=1}^{d}\a_{i}d_{i}$, where $\brac{d_{i}}_{i=1}^{d}$ are
the columns of $\based$. Let $\brac{b_{i}}_{i=1}^{d}$ be the columns
of $\base$. It remains to show that $\a_{j}\in\ZZ$ for every $j$,
which is the case since 
\[
\a_{j}=\sum_{i=1}^{d}\a_{i}\dl_{i,j}=\sum_{i=1}^{d}\a_{i}\left\langle b_{i},d_{j}\right\rangle =\dbrac{b_{i},\sum_{i=1}^{d}\a_{i}d_{j}}=\left\langle b_{i},y\right\rangle \in\ZZ,
\]
where in the second equality we applied Lemma \ref{lem: dual basis properties},
and in the final inclusion we used $b_{i}\in\lat_{\base}$ and $y\in\dual{\lat_{\base}}$. 

For the reverse inclusion $\lat_{\based}\subseteq\dual{\lat_{\base}}$,
it is sufficient to check that $\based\subseteq\dual{\lat_{\base}}$,
namely that $\left\langle x,d_{j}\right\rangle \in\ZZ$ for every
$x\in\lat_{\base}$ and $j=1,\ldots,d$. Indeed, write $x=\sum_{i=1}^{d}\a_{i}b_{i}$
where $\a_{i}\in\ZZ$. Then again by Lemma \ref{lem: dual basis properties},
\[
\left\langle x,d_{j}\right\rangle =\dbrac{\sum_{i=1}^{d}\a_{i}b_{i},d_{j}}=\sum_{i=1}^{d}\a_{i}\left\langle b_{i},d_{j}\right\rangle =\sum_{i=1}^{d}\a_{i}\dl_{i,j}=\a_{j}\in\ZZ.\tag*{\qedhere}
\]
\end{proof}
\begin{cor}
\label{cor: the dual lattice is a lattice}The dual lattice $\dual{\lat}$
is a $d$-lattice, that spans the same linear subspace as $\lat$. 
\end{cor}

\begin{cor}
\label{cor: dual reverses volums}One has that $\dual{\brac{\dual{\lat}}}=\lat$,
and $\covol{\dual{\lat}}=\covol{\lat}^{-1}$. 
\end{cor}

\begin{proof}
Check that $\based\brac{\based^{\transpose}\based}^{-1}=\base$ 
and that $\brac{\vbrac{\det\brac{\based^{\transpose}\based}}}^{1/2}=1/\brac{\vbrac{\det\brac{\base^{\transpose}\base}}}^{1/2}$. 
\end{proof}
\begin{example}
\label{exa: Dual of integral lattice}The dual of $\ZZ^{n}$ is $\ZZ^{n}$,
but in general the dual of an integral (and even primitive) lattice
must not be integral. For example, the dual of $\ZZ(1,1)$  is $\ZZ(\frac{1}{2},\frac{1}{2})$.
More generally, when $\lat<\ZZ^{n}$, then the entries of $\dual{\lat}$
are in the ring $\ZZ\sbrac{\frac{1}{\covol{\lat}^{2}}}$. To see
this, recall that (Proposition \ref{prop: dual lattice basis}) if
$\base$ is a basis for $\lat$, then $\base\brac{\base^{\transpose}\base}^{-1}$
is a basis for $\dual{\lat}$. Since $\brac{\base^{\transpose}\base}^{-1}=\text{adj}\left(\base\right)/\det\brac{\base^{\transpose}\base}=\text{adj}\left(\base\right)/\covol{\lat}^{2}$
where $\text{adj}\left(\base\right)$ is the adjucate matrix of $\base$,
and $\text{adj}\left(\base\right)$ is integral since $\base$ is,
we get that the entries of $\brac{\base^{\transpose}\base}^{-1}$
(and therefore also the entries of $\base\brac{\base^{\transpose}\base}^{-1}$)
are in $\ZZ\sbrac{\frac{1}{\covol{\lat}^{2}}}$. 
\end{example}

Consider the following two spaces: 
\[
\widetilde{\latspace{d,n}}:=\gl n\left(\RR\right)/\left(\left[\begin{smallmatrix}\gl d\left(\ZZ\right) & \RR^{d,n-d}\\
0_{n-d\times d} & \gl{n-d}\left(\RR\right)
\end{smallmatrix}\right]\right)=\text{space of \ensuremath{d}-lattices inside \ensuremath{\RR^{n}}}
\]
\[
\widetilde{\pairspace{d,n}}:=\gl n\left(\RR\right)/\left(\left[\begin{smallmatrix}\gl d\left(\ZZ\right) & \RR^{d,n-d}\\
0_{n-d\times d} & \gl{n-d}\left(\ZZ\right)
\end{smallmatrix}\right]\right)=\begin{array}{c}
\text{space of pairs \ensuremath{\left(\lat,L\right)} where \ensuremath{\lat} a \ensuremath{d}-lattice in \ensuremath{\RR^{n}}}\\
\text{and \ensuremath{L} is an \ensuremath{n-d} lattice in \ensuremath{\perpen{V_{\lat}}}}
\end{array}
\]
Let us justify why these quotients are indeed the spaces of lattices
we claim they are. A coset of a group element $g$ inside $\widetilde{\latspace{d,n}}$
corresponds to the $d$-lattice spanned by the first $d$ columns
of $g$. A coset of a group element $g$ inside $\widetilde{\pairspace{d,n}}$
corresponds to the pair $\left(\lat,L\right)$, where $\lat$ is the
$d$-lattice spanned by the first $d$ columns of $g$, and $L$ is
the lattice spanned by the orthogonal projections of the last $n-d$
columns of $g$ to $\perpen{V_{\lat}}$. Note that since the columns
of $g$ are independent, and the first $d$ columns span $V_{\lat}$,
then the projections of the last $n-d$ columns to $\perpen{V_{\lat}}$
must be independent; in particular, $L$ is an $\left(n-d\right)$-lattice.
Proving that these two identifications are well defined and bijective
is an easy exercise. 
\begin{prop}
\label{prop: Dual is measure preserving diffeo of L_d,n}The map $\lat\mapsto\dual{\lat}$
is a measure preserving diffeomorphism from $\widetilde{\latspace{d,n}}$
to itself.
\end{prop}

\begin{proof}
Write a matrix $g\in\gl n\left(\RR\right)$, as two rectangular matrices
$g=\left[\base|\basec\right]$, where $\base_{n\times d}$ and $\basec_{n-d\times n}$.
Define a map from $\gl n\left(\RR\right)$ to itself by $\left[\base|\basec\right]\mapsto\left[\base\brac{\base^{\transpose}\base}^{-1}|\basec\right]$.
This map is well defined modulo the subgroup $\left[\begin{smallmatrix}\gl d\left(\ZZ\right) & \RR^{d,n-d}\\
0_{n-d\times d} & \gl{n-d}\left(\RR\right)
\end{smallmatrix}\right]$, and so it descends to a map from $\widetilde{\latspace{d,n}}$ to
itself . The latter maps the lattice represented by the coset of
$g$, $\lat=\lat_{\base}$, to  $\lat_{\base\brac{\base^{\transpose}\base}^{-1}}$,
which is $\dual{\lat}$, according to Proposition \ref{prop: dual lattice basis}.
Hence, the descended map from $\widetilde{\latspace{d,n}}$ to itself
is $\lat\mapsto\dual{\lat}$. Being an involution by Corollary \ref{cor: dual reverses volums},
this map is bijective and measure preserving. It is a diffeomorphism
(into its image, which is the whole space), since it is obtained from
a map on $\gl n\left(\RR\right)$ which is a diffeomorphism, because
it is algebraic and therefore smooth, and since it is its own inverse,
hence also its inverse is smooth.
\end{proof}
Following Proposition \ref{prop: Dual is measure preserving diffeo of L_d,n},
one expects that $\left(\lat,L\right)\mapsto\left(\dual{\lat},\dual L\right)$
is a measure preserving diffeomorphism of $\widetilde{\pairspace{d,n}}$.
Proving this fact will require the following lemma. 

\begin{lem}
\label{lem: Ortho-dual}Let $\base_{n\times d}$ and $\basec_{n\times n-d}$
such that $\base\cup\basec$ is an independent set, and let $\pi:\RR^{n}\to\perpen{V_{\base}}$
be the orthogonal projection. Let $\based_{n\times n-d}$ denote the
last $n-d$ columns of the basis $\dual{\sbrac{\base|\basec}}$. Then
$\lat_{\based}=\dual{\brac{\lat_{\pi\left(\basec\right)}}}$. 
\end{lem}

\begin{proof}
We want to show that $\lat_{\based}$ is the dual of $\lat_{\pi\left(\basec\right)}$.
First, since the set $\basec$ is independent of $V_{\base}$, then
the set $\pi\brac{\basec}$ is an independent set inside $\perpen{V_{\base}}$.
Being of size $\vbrac{\pi\left(\basec\right)}=\vbrac{\basec}=n-d=\dim\left(\perpen{V_{\base}}\right)$,
then $\pi\left(\basec\right)$ also spans $\perpen{V_{\base}}$.
Moreover, according to Lemma \ref{lem: dual basis properties} and
the fact that $\dual{\sbrac{\base|\basec}}=\sbrac{*|\based}$, we
have that the columns of $\base$ and $\based$ are orthogonal, meaning
that  $\lat_{\based}\subset\perpen{V_{\base}}$. All in all $V_{\based}=V_{\lat_{\pi\brac{\basec}}}=\perpen{V_{\base}}$.
To show $\lat_{\based}\subseteq\lat_{\pi\brac{\basec}}$, it suffices
to show that $\based\subseteq\lat_{\pi\brac{\basec}}$. Namely, that
for every $c\in\basec$ and $y\in\based$ it holds that $\left\langle \pi\left(c\right),y\right\rangle \in\ZZ$.
Indeed, 
\[
\left\langle \pi\left(c\right),y\right\rangle =\left\langle c,y\right\rangle \in\ZZ
\]
where the equality is due to the fact that $y\in\perpen{V_{\base}}$
and the inclusion holds since $c\in\base|\basec$ and $y\in\dual{\left(\base|\basec\right)}$.
This shows that $\lat_{\based}\subseteq\dual{\brac{\lat_{\pi\left(\basec\right)}}}$.
For the reverse inclusion, let $y\in\dual{\brac{\lat_{\pi\left(\basec\right)}}}$.
Since in particular $y\in V_{\based}$, so $y=\sum\a_{i}d_{i}$. We
need to show that $\a_{i}$ are integers. indeed, 
\[
\a_{i}=\sum\a_{i}\dl_{i,j}=\sum\a_{i}\dbrac{d_{i},\pi\brac{c_{j}}}=\dbrac{\sum\a_{i}d_{i},\pi\brac{c_{j}}}=\dbrac{y,\pi\brac{c_{j}}}\in\ZZ.
\]
\end{proof}

\begin{prop}
\label{prop: All that Dual... and P_d,n}The three maps 
\[
\begin{array}{ccc}
\left(\lat,L\right)\mapsto\left(\dual{\lat},L\right), & \left(\lat,L\right)\mapsto\left(\lat,\dual L\right), & \left(\lat,L\right)\mapsto\left(\dual{\lat},\dual L\right)\end{array}
\]
from $\widetilde{\pairspace{d,n}}$ to itself are measure preserving
diffeomorphisms.
\end{prop}

\begin{proof}
Consider first the map $\gl n\left(\RR\right)\to\gl n\left(\RR\right)$
defined in the proof of Proposition \ref{prop: Dual is measure preserving diffeo of L_d,n}.
This map is well defined modulo the subgroup $\left[\begin{smallmatrix}\gl d\left(\ZZ\right) & \RR^{d,n-d}\\
0_{n-d\times d} & \gl{n-d}\left(\ZZ\right)
\end{smallmatrix}\right]$, and so it descends to a map from $\widetilde{\pairspace{d,n}}$
to itself. As such, according to the explanation about the identification
between cosets of this subgroup in $\sl n\left(\RR\right)$ and elements
$\left(\lat,L\right)$ in $\widetilde{\pairspace{d,n}}$, the descended
map from $\widetilde{\pairspace{d,n}}$ to itself is $\left(\dual{\lat},L\right)\mapsto\left(\dual{\lat},L\right)$.
Similarly to the proof of Proposition \ref{prop: Dual is measure preserving diffeo of L_d,n},
the map is a measure preserving diffeomorphism. 

For the map $\left(\lat,L\right)\mapsto\left(\lat,\dual L\right)$,
consider the map from $\gl n\left(\RR\right)$ to itself given by
$g=\sbrac{\base|\basec}\mapsto\sbrac{\base|\based}$, where $\base=\base_{n\times d}$,
$\basec=\basec_{n\times n-d}$ and $\based$ is the last $n-d$ columns
of the basis $\dual g=\dual{\sbrac{\base|\basec}}$. This map is algebraic
and also an involution, because by taking the last $n-d$ columns
of the basis $\dual{\brac{\dual{\sbrac{\base|\basec}}}}=\base|\basec$
we recover $\basec$. It is defined modulo the subgroup $\left[\begin{smallmatrix}\gl d\left(\ZZ\right) & \RR^{d,n-d}\\
0_{n-d\times d} & \gl{n-d}\left(\ZZ\right)
\end{smallmatrix}\right]$, and according to Lemma \ref{lem: Ortho-dual}, it descends to the
map $\left(\lat,L\right)\mapsto\left(\lat,\dual L\right)$ from $\widetilde{\pairspace{d,n}}$
to itself. Again, it is a measure preserving diffeomorphism. The map
$\left(\lat,L\right)\mapsto\left(\dual{\lat},\dual L\right)$ is obtained
similarly from $g=\sbrac{\base|\basec}\mapsto\sbrac{\base\brac{\base^{\transpose}\base}^{-1}|\based}$,
where $\base,\basec,\based$ are as before. 
\end{proof}

\subsubsection*{Spaces of unimodular lattices }

The space $\latspace{d,n}=\sl n\left(\RR\right)/\left(\left[\begin{smallmatrix}\sl d\left(\ZZ\right) & \RR^{d,n-d}\\
0_{n-d\times d} & \sl{n-d}\left(\RR\right)
\end{smallmatrix}\right]\times A^{\prime}\right)$, where 
\[
A^{\prime}=\left\{ \left[\begin{smallmatrix}\a^{\frac{1}{d}}\idmat d & 0_{d\times n-d}\\
0_{n-d\times d} & \a^{-\frac{1}{n-d}}\idmat{n-d}
\end{smallmatrix}\right]:\a>0\right\} ,
\]
is the space of oriented $d$-lattices in $\RR^{n}$, up to homothety.
Similarly, the space $\pairspace{d,n}=\sl n\left(\RR\right)/\left(\left[\begin{smallmatrix}\sl d\left(\ZZ\right) & \RR^{d,n-d}\\
0_{n-d\times d} & \sl{n-d}\left(\ZZ\right)
\end{smallmatrix}\right]\times A^{\prime}\right)$ is the space of pairs $\left(\lat,L\right)$ of oriented lattices
satisfying $\covol{\lat}\covol{\qlat}=1$, where $\lat$ is a $d$-lattice
and $L$ is an $\brac{n-d}$-lattice in $\perpen{V_{\lat}}$. More
accurately, it is the space of equivalence classes of such pairs,
modulo the equivalence relation $\left(\lat^{\prime},L^{\prime}\right)\sim\left(\lat,L\right)$
iff there exists $\a>0$ such that $\left(\lat^{\prime},L^{\prime}\right)=\left(\a^{\frac{1}{d}}\lat,L^{-\frac{1}{n-d}}\right)$. 

These spaces, unlike $\widetilde{\latspace{d,n}}$ and $\widetilde{\pairspace{d,n}}$,
have finite volume. 
\begin{prop}
\label{rem: dual of unimodular}Propositions \ref{prop: Dual is measure preserving diffeo of L_d,n}
and \ref{prop: All that Dual... and P_d,n} hold also when replacing
$\widetilde{\latspace{d,n}}$ and $\widetilde{\pairspace{d,n}}$ by
$\latspace{d,n}$ and $\pairspace{d,n}$. 
\end{prop}

Indeed, the dual of a unimodular lattice is also unimodular, by Corollary
\ref{cor: dual reverses volums}; so the map $\lat\mapsto\dual{\lat}$
is an involution of $\latspace{d,n}$, and similarly the maps in Proposition
\ref{prop: All that Dual... and P_d,n} are involutions of $\pairspace{d,n}$.
The proofs for the adaptations of Propositions \ref{prop: Dual is measure preserving diffeo of L_d,n}
and \ref{prop: All that Dual... and P_d,n} to unimodular lattices
are obtained by adjusting the proofs of these propositions: the appearances
of $\gl{}$ are replaced by $\sl{}$, and the ambient group should
be quotiented by $A^{\prime}$ as well. 
\begin{rem}
Since the spaces $\latspace{d,n}$ and $\pairspace{d,n}$ are unbounded
but of finite volume, it is desireble to have a criterion for determining
when a set is compact, or alternatively, when does a sequence of elements
in the space diverge to infinity. For the more familiar space of full
unimodular lattices in $\RR^{d}$, $\sl d\left(\RR\right)/\sl d\left(\ZZ\right)$,
which is also non-compact but of finite measure, the answer is provided
by Mahler's Compactness Criterion \cite[Chapter V.3]{Bekka_Mayer}.
The latter states that a set is compact if and only if there exists
$\dl>0$ such that all the lattices in this set have shortest vector
of length $>\dl$. Equivalently, a sequence of lattices $\{\lat_{m}\}$
diverges if and only if the lengths of the shortest vectors in $\{\lat_{m}\}$
is a squence of positive real numbers that converges to zero. The
purpose of the following is to state an analogous criterion for compactness
in the spaces $\latspace{d,n}$ and $\pairspace{d,n}$. For every
equivalence class $\unilat{\lat_{0}}\in\latspace{d,n}$ (resp.\ $\unilat{(\lat_{0},\qlat_{0})}\in\pairspace{d,n}$),
a \emph{unimodular representative} is the unique representative $\lat\in\unilat{\lat_{0}}$
of covolume one (resp.\ $(\lat,\qlat)\in\unilat{(\lat_{0},\qlat_{0})}$
such that $\lat$ and $\qlat$ are of covolume one).

\end{rem}

\begin{prop}
A subset $\latset$ of $\latspace{d,n}$ is bounded iff there is some
$\dl>0$ such that for every $\unilat{\lat_{0}}\in\latset$, its unimodular
representative $\lat\in\unilat{\lat_{0}}$ has the property that the
shortest vector of $\lat$ is of length $\geq\dl$. 

A subset $\pairset$ of $\pairspace{d,n}$ is bounded iff there is
some $\dl>0$ such that for every $\unilat{(\lat_{0},\qlat_{0})}\in\pairset$,
its unimodular representative $(\lat,\qlat)\in\unilat{(\lat_{0},\qlat_{0})}$
has the property that both shortest vectors of $\lat$ and $\qlat$
are of length $\geq\dl$. 
\end{prop}

\begin{proof}
We will prove the claim for $\pairspace{d,n}$, since case of $\latspace{d,n}$
is similar. It is not hard to see (e.g., Remark \ref{rem: true quotients L_d,n and P_d,n}
in the present paper) that $\pairspace{d,n}=\wc/G^{\prime\prime}(\mathbb{Z})$
where
\[
\wc=\so n\left(\RR\right)\left[\begin{smallmatrix}P_{d} & 0\\
0 & P_{n-d}
\end{smallmatrix}\right]\subset\sl n(\RR),\quad G^{\dprime}=\left[\begin{smallmatrix}\sl d\left(\RR\right) & 0_{d,n-d}\\
0_{n-d,d} & \sl{n-d}\left(\RR\right)
\end{smallmatrix}\right]<\sl n(\RR),
\]
and $P_{i}$ is the group of $i\times i$ upper triangular matrices
with positive diagonal elements and determinant one. Moreover, there
exists a construction for a fundamental domain $\groupfund i\subset\sl i(\RR)$
representing $\sl i(\RR)/\sl i(\ZZ)$ and contained in a \emph{Siegel
set} (\cite{Bekka_Mayer}), and a subset $K^{\prime}$ of $\so n\left(\RR\right)$,
such that (Proposition \cite[Prop. 8.1]{HK_WellRoundedness}) the
set $K^{\prime}(\groupfund d\times\groupfund{n-d})\subset\wc$, where
$\groupfund d\times\groupfund{n-d}$ stands for $\left[\begin{smallmatrix}\groupfund d & 0_{d,n-d}\\
0_{n-d,d} & \groupfund{n-d}
\end{smallmatrix}\right]\subset G^{\dprime}$, is a fundamental domain representing $\pairspace{d,n}$ with the
property that the restriction of the natural projection $\pi:Q\to\pairspace{d,n}$
to $\overline{K^{\prime}(\groupfund d\times\groupfund{n-d})}$ is
proper. As a result, $\pairset$ is bounded iff $\pi|_{K^{\prime}(\groupfund d\times\groupfund{n-d})}^{-1}(\pairset)$
is bounded.

Since $K^{\prime}$ is bounded, and by definition of the Siegel sets
(or the construction of $\groupfund d,\groupfund{n-d}$), a set that
is contained inside $K^{\prime}(\groupfund d\times\groupfund{n-d})$
is unbounded if and only if its projection to the diagonal subgroup
in $G^{\dprime}$ is unbounded. Again by the construction of the Siegel
sets, an element in the projection of $K^{\prime}(\groupfund d\times\groupfund{n-d})$
 is of the form $\text{diag}(a_{1},\cdots,a_{d})\times\text{diag}(b_{1},\cdots,b_{n-d})$,
where $0<a_{1}\prec a_{2}\cdots\prec a_{d}$ ; $0<b_{1}\prec b_{2}\cdots\prec b_{n-d}$
and $a_{1}\cdots a_{d}=1=b_{1}\cdots b_{n-d}$. Notice that $a_{1}$
(resp.\ $b_{1}$) is bounded from below iff $a_{d}=\frac{1}{a_{1}\cdots a_{d-1}}\succ\frac{1}{a_{1}^{d}}$
(resp.\ $b_{n-d}=\frac{1}{b_{1}\cdots b_{n-d}}\succ\frac{1}{b_{1}^{n-d}}$)
is bounded from above. As a consequence, $\pairset$ is bounded if
and only if the entries $a_{1}$ and $b_{1}$ in the diagonal components
of $\pi|_{K^{\prime}(\groupfund d\times\groupfund{n-d})}^{-1}(\pairset)$
are bounded from below by some $\epsilon>0$. 

Let $\lat^{1}$ and $\qlat^{1}$ be the rank one lattices spanned
by the shortest elements in $\lat$ and $\qlat$ respectively. According
to Proposition \ref{prop: explicit RI coordinates of g} parts $(ii),(ii)^{\#},(iii),(iii)^{\#}$
, and using the construction of a Siegel set (that relies on the
notion of a \emph{Siegel basis} for a lattice, in which the first
element is a shortest vector in a lattice), the boundedness from below
of $a_{1}$ and $b_{1}$ is equivaent to boundedness from below of
the covolmes of $\lat^{1}$ and $\qlat^{1}$, which are clearly the
lengths of the shortest vectors in $\lat$ and $\qlat$ respectively. 
\end{proof}

\subsubsection*{Relation between symmetries of a lattice and of its dual}

A \emph{shape} of a lattice (sometimes referred to as a \emph{type}
of a lattice) is its similarity class modulo rotation and rescaling.
It is a very common parameter to study in the context of lattices,
and appears e.g. in crystallography and the theory of periodic tilings.
 If two lattices in $\RR^{n}$ have the same shape, then in particular
they are invariant under the same symmetries, meaning the same orthogonal
transformations of $\RR^{n}$. The opposite does not hold, e.g. the
lattices $\ZZ e_{1}\oplus\ZZ\a e_{2}$ and $\ZZ e_{1}\oplus\ZZ\b e_{2}$
for $1<\a<\b$ have different shapes, but both are invariant only
under the orthogonal transformations $\pm\left[\begin{smallmatrix}1 & 0\\
0 & 1
\end{smallmatrix}\right],\pm\left[\begin{smallmatrix}1 & 0\\
0 & -1
\end{smallmatrix}\right]\in\ort 2\left(\RR\right)$. Below we observe the surprising fact that even though a lattice
and its dual in general do not have the same shape (e.g., $\ZZ e_{1}\oplus\ZZ\a e_{2}$
and $\ZZ e_{1}\oplus\ZZ e_{2}/\a$ are dual), they do share the same
group of symmetries. 
\begin{defn}
Let $1\leq d\leq n$, and a $d$-lattice $\lat$. The finite group
$\sym\brac{\lat}:=\left\{ g\in\ort{}\brac{V_{\lat}}:g\lat=\lat\right\} $
is called the \emph{symmetry group} of $\lat$.  
\end{defn}

\begin{claim}
\label{lem: symmetry group of dual lattice}The symmetry groups of
$\lat$ and $\dual{\lat}$ are identical. 
\end{claim}

\begin{proof}
Let $B_{n\times d}$ be a matrix for $\lat$, so that $D=B\left(B^{\transpose}B\right)^{-1}$
is a matrix for $\dual{\lat}$. Assume $k\in\ort n\brac{\RR}$ is
in the symmetry group of $\lat$, i.e. there exists $\ga\in\gl d\left(\ZZ\right)$
such that $kB=B\ga$. It follows that $kD=B\ga^{-\transpose}$, since
$kD=kB\left(B^{\transpose}B\right)^{-1}=kB\left(\left(kB\right)^{\transpose}\left(kB\right)\right)^{-1}$
and this equals 
\[
=B\ga\brac{\left(B\ga\right)^{\transpose}\left(B\ga\right)}^{-1}=B\ga\brac{\ga^{\transpose}B^{\transpose}B\ga}^{-1}=B\ga\cdot\ga^{-1}\brac{B^{\transpose}B}^{-1}\ga^{-\transpose}=D\ga^{-\transpose}.
\]
This means that $k$ is also in the symmetry group of $\dual{\lat}$,
so $\sym\left(\lat\right)\subseteq\sym\brac{\dual{\lat}}\subseteq\sym\brac{\dual{\brac{\dual{\lat}}}}=\sym\left(\lat\right)$,
resulting in $\sym\left(\lat\right)=\sym\brac{\dual{\lat}}$.
\end{proof}
We conclude our introduction to dual lattices with a short discussion
on the successive minima of the dual lattice. 
\begin{thm}[{\cite[Thm. 2.1]{Banaszczyk93}, see also \cite[Chapter VIII.5, Thm. VI]{Cassels71}}]
\label{thm: successive minima of dual lattice}If $\lm_{1},\ldots,\lm_{k}$
are the successive minima of a $k$-lattice and $\lm_{1}^{*},\ldots,\lm_{k}^{*}$
of its dual, then $1\leq\lm_{j}\lm_{k-j+1}^{*}\leq k$ for every $1\leq j\leq k$.
\end{thm}

From this theorem, we conclude the following lemma. Let $C\left(k\right)>0$
be such that every $k$-dimensional lattice $L$ with successive minima
$\lm_{1},\ldots,\lm_{k}$ satisfies that 
\begin{equation}
\frac{1}{C\left(k\right)}\lm_{1}\cdots\lm_{k}\leq\covol L\leq\lm_{1}\cdots\lm_{k};\label{eq: Minkowski successive minima}
\end{equation}
such $C\left(k\right)$ exists by Minkowski's Second Theorem \cite[chapter VIII.2, theorem 1]{Cassels71}. 
\begin{lem}
\label{lem: covolume of sub-lattices of factor and perp}Let $\qlat$
be a $k$-lattice in $\RR^{n}$, with $1\leq k\leq n-1$. Then for
$C\left(\cdot\right)$ as in Formula \ref{eq: Minkowski successive minima},
and every $1\leq j\leq k$,
\[
\frac{\covol{\qlat^{j}}}{C\left(k-j\right)k^{j}}\leq\frac{\covol{\left(\dual{\qlat}\right)^{k-j}}}{\covol{\dual{\qlat}}}\leq C\left(j\right)C\left(k\right)\covol{\qlat^{j}}.
\]
\end{lem}

\begin{proof}
From Theorem \ref{thm: successive minima of dual lattice} and Formula
\ref{eq: Minkowski successive minima} it follows that for every $1\leq j\leq k$,
\[
\covol{\qlat^{j}}\leq\lm_{1}\cdots\lm_{j}\leq\frac{k^{j}}{\dual{\lm_{k}}\cdots\dual{\lm_{k-j+1}}}
\]
\[
\leq\frac{k^{j}}{\dual{\lm_{k}}\cdots\dual{\lm_{k-j+1}}}\cdot\frac{\dual{\lm_{1}}\cdots\dual{\lm_{k}}}{\covol{\dual{\qlat}}}=k^{j}\cdot\frac{\dual{\lm_{1}}\cdots\dual{\lm_{k-j}}}{\covol{\dual{\qlat}}}\leq\frac{k^{j}\cdot C\left(k-j\right)\covol{\brac{\dual{\qlat}}^{k-j}}}{\covol{\dual{\qlat}}},
\]
which proves the left-hand side inequality. The right-hand side inequality
is proved similarly.
\end{proof}

\section{Factor lattices}

The term was coined by Schmidt in \cite{Schmidt_68}, for primitive
integral lattices. We extend it here to general lattices, that are
not necessarily integral; for this, we begin by extending the definition
of primitiveness to lattices that are not necessarily integral. 
\begin{defn}
Assume $\latfull$ is a full lattice in $\RR^{n}$, and $\lat$ is
a $d$-lattice that is contained in $\latfull$. We say that $\lat$
is \emph{primitive} inside (or with respect to) $\latfull$ if $\lat=\latfull\cap V_{\lat}$. 
\end{defn}

Note that when $\latfull=\ZZ^{n}$, this definition agrees with the
standard definition of a primitive lattice. Indeed, in this case,
we will call $\lat$ primitive (and omit the ``w.r.t.\ $\ZZ^{n}$''). 
\begin{defn}
Given a $d$-lattice $\lat$ that is primitive inside $\latfull$,
define the factor lattice of $\lat$ (w.r.t.\ $\latfull$), denoted
$\factorg{\lat}{\latfull}$, as the orthogonal projection of $\latfull$
into the space $\perpen{\brac{\sp{\RR}{\lat}}}$. When $\lat$ is
primitive inside $\ZZ^{n}$, we omit the ``w.r.t.\ $\ZZ^{n}$''
from the name and the notation: we denote $\factor{\lat}$ and refer
to it as the factor lattice of $\lat$.
\end{defn}

For example, the factor lattice of $\ZZ\brac{1,-1}$ is $\ZZ\brac{\frac{1}{2},\frac{1}{2}}$.
A priori, it is not clear that $\factorg{\lat}{\latfull}$ is indeed
a lattice; the following proposition assures us that this is the case. 
\begin{prop}
\label{prop: factor is quotient}For a $d$-lattice $\lat$ that is
primitive inside a full lattice $\latfull$, consider the inner product
on the quotient $\RR^{n}/V_{\lat}$:
\[
\left\langle x+V_{\lat},y+V_{\lat}\right\rangle _{\RR^{n}/V_{\lat}}:=\left\langle \pi(x),\pi(y)\right\rangle ,
\]
where $\pi:\RR^{n}\to\perpen{V_{\lat}}$ is the orthogonal projection
and $\left\langle \cdot,\cdot\right\rangle $ is the standard inner
product on $\RR^{n}$. Then the quotient lattice $\latfull/\lat$
with $\left\langle \cdot,\cdot\right\rangle _{\RR^{n}/V_{\lat}}$
is isometric to the factor lattice $\factorg{\lat}{\latfull}$ with
$\left\langle \cdot,\cdot\right\rangle $. 
\end{prop}

\begin{proof}
The group isomorphism $x+\lat\mapsto\pi(x)$ from $\latfull/\lat$
to $\factorg{\lat}{\latfull}$ is an isometry by definition.
\end{proof}
Then $\factorg{\lat}{\latfull}$ is a lattice, because it is isometric
to one. The above also demonstrates that it is necessary to require
that $\lat$ is primitive inside $\latfull$; otherwise, $\latfull/\lat$
(and therefore $\factorg{\lat}{\latfull}$) is a finite group. 
\begin{prop}
\label{prop: factor covolume}$\covol{\factorg{\lat}{\latfull}}=\covol{\latfull}/\covol{\lat}$. 
\end{prop}

\begin{proof}
since covolumes do not change under rotations, we may assume that
$V_{\lat}=\mathbb{E}_{d}:=\sp{\RR}{e_{1},\ldots,e_{d}}$, where $d=\rank{\lat}$.
Let $\base_{d\times n}$ be a basis for $\lat$; by primitiveness,
it can be completed to a basis $g:=\sbrac{\base|\basec}$ of $\latfull$,
where $g\in\gl n\left(\RR\right)$. Since $\base\subset\mathbb{E}_{d}$,
the matrix $g$ is of the form $g=\left[\begin{smallmatrix}g_{1} & *\\
0 & g_{2}
\end{smallmatrix}\right]$ with $g_{1}\in\gl d\left(\RR\right)$ and $g_{2}\in\gl{n-d}\left(\RR\right)$.
Clearly $\lat=\sbrac{g_{1}|0_{d\times n}}\ZZ^{n}$ and therefore $\covol{\lat}^{2}=\det\brac{\sbrac{\begin{smallmatrix}g_{1} & 0_{d\times n-d}\end{smallmatrix}}\left[\begin{smallmatrix}g_{1}\\
0_{n-d\times d}
\end{smallmatrix}\right]}=\det\left(g_{1}\right)^{2}$, so $\covol{\lat}=\det(g_{1})$. Note that the projection of $g$
to $\perpen{\mathbb{E}_{d}}$ is $\left[\begin{smallmatrix}0_{d\times n-d}\\
g_{2}
\end{smallmatrix}\right]$, so the projection of $\lat_{g}=\latfull$ to $\perpen{\mathbb{E}_{d}}=\perpen{V_{\lat}}$
is the lattice spanned by $\left[\begin{smallmatrix}0_{d\times n-d}\\
g_{2}
\end{smallmatrix}\right]$. By definition this lattice is $\factorg{\lat}{\latfull}$. We conclude
that $\covol{\factorg{\lat}{\latfull}}^{2}=\det\brac{\left[\begin{smallmatrix}0_{d\times n-d}\\
g_{2}
\end{smallmatrix}\right]\sbrac{\begin{smallmatrix}0_{n-d\times d} & g_{2}\end{smallmatrix}}}=\det\left(g_{2}\right)^{2}$, namely $\covol{\factorg{\lat}{\latfull}}=\det\left(g_{2}\right)$.
Finally, $\covol{\latfull}=\det(g)=\det(g_{1})\cdot\det(g_{2})$. 
\end{proof}
In the case where $\latfull=\ZZ^{n}$, namely of primitive integral
lattices, the concepts of a dual, a factor, and the orthogonal lattice
are related in the following way. 
\begin{prop}[\cite{Schmidt_98}]
\label{prop: dual+factor =00003D orthogonal}For $\lat$ primitive,
$\dual{\brac{\factor{\lat}}}=\perpen{\lat}$. 
\end{prop}

\begin{proof}
Let $\base_{n\times d}^{\prime}$ be a basis for $\lat$, and use
primitiveness to complete it to a basis $\base_{n\times n}$ of $\ZZ^{n}$;
so $\base\in\sl n\left(\ZZ\right)$. Let $\based_{n\times n}$ be
the dual basis of $\base$, namely $\based=\base\brac{\base^{\transpose}\base}^{-1}=\base^{-\transpose}$,
which is also in $\sl n\left(\ZZ\right)$. Let $\based_{n\times n-d}^{\prime}$
denote the last $n-d$ columns of $\based$. Since $\base^{\transpose}\based=\idmat n$,
then $\base^{\prime}$ is orthogonal to $\based^{\prime}$; in particular,
$\lat_{\based^{\prime}}\subseteq\perpen{V_{\base^{\prime}}}\cap\ZZ^{n}=\perpen{V_{\lat}}\cap\ZZ^{n}$,
so by definition of orthogonal lattice $\lat_{\based^{\prime}}\subseteq\perpen{\lat}$.
But this inclusion is in fact an equality, since on both sides there
are primitive lattices in the same space $V_{\lat}$ and of the same
rank. 

We will conclude by showing that $\lat_{\based^{\prime}}$ is also
$\dual{\brac{\factor{\lat}}}$. Indeed, if $\pi:\RR^{n}\to\perpen{V_{\lat}}$
is the orthogonal projection, then $\pi(\base)=\pi(\base^{\prime})$
 is a basis to $\factor{\lat}$, by definition. By Lemma \ref{lem: Ortho-dual},
$\lat_{\based^{\prime}}=\dual{\left(\lat_{\pi\left(\base^{\prime}\right)}\right)}$,
which concludes the proof. 
\end{proof}
\begin{cor}
\label{cor: orthogonal covolume}For $\lat$ primitive, $\covol{\perpen{\lat}}=\covol{\lat}$. 
\end{cor}

\begin{proof}
this is a direct consequence of Corollary \ref{cor: dual reverses volums},
Proposition \ref{prop: factor covolume} Proposition \ref{prop: dual+factor =00003D orthogonal},
while noticing that $\covol{\ZZ^{n}}=1$. 
\end{proof}
\begin{rem}
One could wonder if the result of Proposition \ref{prop: dual+factor =00003D orthogonal}
can be extended to general lattices, that are not necessarily primitive,
or even integral. For example, it is quite natural to extend the definition
of $\perpen{\lat}$ for general lattices in a similar fashion to what
we did with $\factorg{\lat}{\latfull}$, namely for a lattice $\lat$
that is primitive inside a full lattice $\latfull$, define the orthogonal
lattice of $\lat$ w.r.t.\ $\latfull$, as $\latfull\cap\perpen{\brac{\sp{\RR}{\lat}}}$.
Say we denote it $\perpeng{\lat}{\latfull}$. Now the question becomes,
is it true that the dual of $\factorg{\lat}{\latfull}$ is $\perpeng{\lat}{\latfull}$.
The answer to this is no! Let us consider two counter examples. 
\begin{enumerate}
\item Consider the lattice $\latfull$ spanned by $\left[\begin{smallmatrix}1 & \sqrt{2}\\
0 & 1
\end{smallmatrix}\right]$, and the lattice $\lat=\sp{\ZZ}{\left[\begin{smallmatrix}1\\
0
\end{smallmatrix}\right]}$ which is primitive inside it. Then by definition $\perpeng{\lat}{\latfull}=\latfull\cap\perpen{V_{\lat}}=\latfull\cap\sp{\RR}{\left[\begin{smallmatrix}0\\
1
\end{smallmatrix}\right]}=\left\{ 0\right\} .$ 
\item Take $\latfull:=\ZZ e_{1}\oplus\ZZ\a e_{2}$ for some $\a>0$. Clearly
$\latfull$ has covolume $\a$. The lattice $\lat=\ZZ e_{1}$ is primitive
inside $\latfull$, and has covolume $1$. Then $\perpeng{\lat}{\latfull}=\ZZ\a e_{2}$
has covolume $\a$, and in particular it cannot be that $\perpeng{\lat}{\latfull}$
is the dual of $\factorg{\lat}{\latfull}$, because otherwise it
would have covolume $\brac{\covol{\latfull}/\covol{\lat}}^{-1}=\a^{-1}$. 
\end{enumerate}
The first example shows that the definition of $\perpeng{\lat}{\latfull}$
is in a sense meaningless; the second example shows us that $\perpeng{\lat}{\latfull}$
being the dual of $\factorg{\lat}{\latfull}$ fails even when $\latfull$
is integral (which happens when we take $\a\in\ZZ$ in the second
example). 

What could be said in general about the dual of $\factorg{\lat}{\latfull}$
then? (Here $\lat<\latfull<\RR^{n}$, with $\rank{\lat}<\rank{\latfull}\leq n$,
and $\latfull$ must not necessarily be a full lattice). According
to Lemma \ref{lem: Ortho-dual}, we know that a basis for $\dual{\brac{\factorg{\lat}{\latfull}}}$
is obtained in the last $n-d$ columns of the matrix $\dual{\sbrac{\base|\basec}}$,
where $\base_{n\times d}$ is a basis for $\lat$ such that $\sbrac{\base|\basec}$
is a basis for $\latfull$. Now, even though $\dual{\brac{\factorg{\lat}{\latfull}}}$
needs not be integral when $\latfull$ is not $\ZZ^{n}$ (i.e. we
are not in the case of Proposition \ref{prop: dual+factor =00003D orthogonal}),
it is ``almost integral'' when $\latfull<\ZZ^{n}$. Indeed, in this
case $\dual{\brac{\factorg{\lat}{\latfull}}}$ is defined over the
ring $\ZZ\sbrac{\frac{1}{\covol{\latfull}^{2}}}$, by a similar argument
to the one in Example \ref{exa: Dual of integral lattice}. 

$ $
\end{rem}

The space of factor lattices of $d$-lattices in $\RR^{n}$ is the
following:
\[
\factor{\widetilde{\latspace{d,n}}}:=\gl n\left(\RR\right)/\left[\begin{smallmatrix}\gl d\left(\RR\right) & \RR^{d,n-d}\\
0_{n-d\times d} & \gl{n-d}\left(\ZZ\right)
\end{smallmatrix}\right],
\]
since, given a coset of $g\in\gl n\left(\RR\right)$, the projection
of the last $n-d$ columns to the space that is orthogonal to the
span of the first $d$ columns, is the factor lattice of the lattice
spanned by the first $d$ columns, that is clearly an element in $\widetilde{\latspace{d,n}}$. 

It is easy to see that there is a natural bijection between the spaces
$\factor{\widetilde{\latspace{d,n}}}$ and $\widetilde{\latspace{n-d,n}}$:
the factor lattice of a $d$-lattice in $\RR^{n}$ (w.r.t. any full
lattice) is an $\left(n-d\right)$-lattice in $\RR^{n}$; conversely,
if $L\in\widetilde{\latspace{n-d,n}}$, then by taking any $d$-lattice
$\lat$ in the space $\perpen{V_{L}}$, we have that $L$ is the factor
lattice of $\lat$ w.r.t. the full lattice $\lat\oplus L$. In fact,
this bijection is a diffeomorphism between the two spaces:
\begin{prop}
This bijection is a measure preserving diffeomorphism between $\factor{\widetilde{\latspace{d,n}}}$
and $\widetilde{\latspace{n-d,n}}$.
\end{prop}

\begin{proof}
Write a matrix $g\in\gl n\left(\RR\right)$, as two rectangular matrices
$g=\left[\base|\basec\right]$, where $\base_{n\times d}$ and $\basec_{n-d\times n}$.
Denote $\mathbf{P}=\idmat n-\base\brac{\base^{\transpose}\base}^{-1}\base^{\transpose}\in\gl n\left(\RR\right)$,
the orthogonal projection to $\perpen{V_{\base}}$, and define a map
from $\gl n\left(\RR\right)$ to itself by $\left[\base|\basec\right]\mapsto\left[\mathbf{P}\basec|\base\right]$.
This map is algebraic and therefore smooth, and since it is defined
modulo the group $\left[\begin{smallmatrix}\gl d\left(\RR\right) & \RR^{d,n-d}\\
0_{n-d\times d} & \gl{n-d}\left(\ZZ\right)
\end{smallmatrix}\right]$, it descends to a map from $\factor{\widetilde{\latspace{d,n}}}$
to $\widetilde{\latspace{n-d,n}}$. Note that the lattice spanned
by the first $n-d$ columns of the resulting matrix is the factor
lattice of $\lat_{\base}$ w.r.t. $\lat_{g}$. In particular, the
resulting matrix lies in the coset of $\gl n\left(\RR\right)$ representing
(in $\widetilde{\latspace{n-d,n}}$) the $\brac{n-d}$\textendash lattice
$\factorg{\lat_{\base}}{\lat_{g}}$, so it is indeed the direction
$\factor{\widetilde{\latspace{d,n}}}\to\widetilde{\latspace{n-d,n}}$
in the bijection we described. 

Conversely, consider $g=\left[\basec|\base\right]\in\gl n\left(\RR\right)$
with $\base_{n\times d}$ and $\basec_{n-d\times n}$ and define the
map $\left[\basec|\base\right]\mapsto\left[\base|\mathbf{P}\basec\right]$.
Again, it is smooth since it is algebraic, and it is defined modulo
the group $\left[\begin{smallmatrix}\gl{n-d}\left(\ZZ\right) & \RR^{n-d,d}\\
0_{d\times n-d} & \gl d\left(\RR\right)
\end{smallmatrix}\right]$ so it descends to a map from $\widetilde{\latspace{n-d,n}}$ to $\factor{\widetilde{\latspace{d,n}}}$.
Again we note that the lattice spanned by the last $n-d$ columns
of the resulting matrix is the factor lattice for the lattice spanned
by the first $d$ columns w.r.t. $\lat_{g}$, so this is indeed the
direction $\factor{\widetilde{\latspace{d,n}}}\leftarrow\widetilde{\latspace{n-d,n}}$
in the bijection we described. 

Finally, let us check that these two maps are the inverses of one
another:
\[
\left[\base|\basec\right]\mapsto\left[\mathbf{P}\basec|\base\right]\to\left[\base|\mathbf{P}\mathbf{P}\basec\right]=\left[\base|\mathbf{P}\basec\right],
\]
where the equality is because  $\mathbf{P}^{2}=\mathbf{P}$. Now
it is only left to observe that $\left[\base|\basec\right]\equiv\left[\base|\mathbf{P}\basec\right]$
modulo $\left[\begin{smallmatrix}\gl{n-d}\left(\ZZ\right) & \RR^{n-d,d}\\
0_{d\times n-d} & \gl d\left(\RR\right)
\end{smallmatrix}\right]$.
\end{proof}
\bibliographystyle{alpha}
\phantomsection\addcontentsline{toc}{section}{\refname}\bibliography{bib_for_dlattices}

\end{document}